\documentclass[12pt]{amsart}
\usepackage[english]{babel}
\usepackage[latin1]{inputenc}
\usepackage{hyperref}
 \usepackage{stmaryrd}
\usepackage{amssymb,amsfonts,amsmath,amsthm}
\usepackage{graphicx}
\usepackage[all]{xy}
\usepackage{enumerate}

\usepackage{xcolor}
\newcommand{\SortNoop}[1]{}

\newcommand{\pt}{\forall}

\newcommand{\plonge}{\hookrightarrow}
\newcommand{\then}{\Rightarrow}

\newcommand{\laction}{\curvearrowright}
\newcommand{\ol}[1]{\overline{#1}}

\newcommand{\mc}[1]{\mathcal{#1}}
\newcommand{\wh}[1]{\widehat{#1}}

\newcommand{\lie}[1]{\mathfrak{#1}}
\newcommand{\ad}{\mathrm{ad}}

\newcommand{\Out}{\mathrm{Out}}
\newcommand{\Aut}{\mathrm{Aut}}
\newcommand{\Ad}{\mathrm{Ad}}

\newcommand{\benum}{\begin{enumerate}}
\newcommand{\eenum}{\end{enumerate}}

\newcommand{\R}{\mathbb{R}}

\newcommand{\C}{\mathbb{C}}
\newcommand{\Z}{\mathbb{Z}}

\newcommand{\Gtheta}{G_{\theta}}
\newcommand{\Ttheta}{T_{\theta}}

\newcommand{\su}{\mathfrak{su}}

\newcommand{\so}{\mathfrak{so}}

\newcommand{\Hom}{\mathrm{Hom}}
\newcommand{\Ker}{\mbox{Ker}}

\newcommand{\SL}{\mathrm{SL}}
\newcommand{\SU}{\mathrm{SU}}

\newcommand{\GL}{\mathrm{GL}}
\newcommand{\SO}{\mathrm{SO}}

\newcommand{\g}{\mathfrak{g}}
\newcommand{\h}{\mathfrak{h}}

\newcommand{\lp}{\mathfrak{p}}
\newcommand{\la}{\mathfrak{a}}
\newcommand{\laC}{\mathfrak{a}}

\newcommand{\lar}{\mathfrak{a}_{reg}}

\newcommand{\z}{\mathfrak{z}}
\newcommand{\m}{\mathfrak{m}}

\newcommand{\gr}{\g_{reg}}
\newcommand{\tgr}{\widehat{\g}_{reg}}

\newcommand{\mr}{{\m}_{reg}}
\newcommand{\mrs}{{\m}_{reg,ss}}
\newcommand{\lb}{\mathfrak{b}}

\newcommand{\lt}{\mathfrak{t}}
\newcommand{\lu}{\mathfrak{u}}
\newcommand{\lc}{\mathfrak{c}}
\newcommand{\ld}{\lie{d}}

 \newcommand{\p}{\mathfrak{p}}

\newcommand{\shitc}{\left[h\right]_{L,\C}}
\newcommand{\shit}{\left[h\right]}
\newcommand{\ashit}{\left[h_{abs}\right]}
\newcommand{\lshitc}{\left[\chi\right]_\C}
\newcommand{\lshit}{\left[\chi\right]}

\newcommand{\Xet}{X_{\acute{e}t}}
\newcommand{\XET}{(Sch/X)_{\acute{e}t}}
\newcommand{\CET}{(Sch/\C)_{\acute{e}t}}
\newcommand{\Higgs}{\mathbf{Higgs}}
\newcommand{\Higgsr}{\mathbf{Higgs}_{reg}}
\newcommand{\Higgsrt}{\mathbf{Higgs}_{reg,\theta}}
\newcommand{\aHiggs}{\mc{H}iggs}

\newcommand{\stack}[1]{\left[#1\right]}

\newcommand{\Puniv}{P^{univ}}
\newcommand{\guniv}{\gamma^{univ}}
\newcommand{\buniv}{\beta^{univ}}
\newcommand{\Cov}{\mc{C}ov}

\newcommand{\NH}{N_H(\la)}
\newcommand{\Ntheta}{N_{\Gtheta}(\la)}
\newcommand{\CH}{C_H(\la)}
\newcommand{\GT}{\overline{G/T}}
\newcommand{\GN}{\overline{G/N}}

\newcommand{\HC}{\overline{H/C_H(\la)}}
\newcommand{\HN}{\overline{H/N_H(\la)}}
\newcommand{\specc}{\widehat{X}}
\newcommand{\spec}{\widehat{X}}

\newcommand{\specu}{\widehat{U}}

\newtheorem{thm}{Theorem}[section]
\newtheorem{prop}[thm]{Proposition}
\newtheorem{cor}[thm]{Corollary}
\newtheorem{lm}[thm]{Lemma}
\theoremstyle{definition}

\newtheorem{ex}[thm]{Example}
\newtheorem{defi}[thm]{Definition}
\newtheorem{fact}[thm]{Fact}

\newtheorem*{thm*}{Theorem}
\theoremstyle{remark}
\newtheorem{rk}[thm]{Remark}
\begin{document}
\title[Higgs bundles, abelian gerbes and cameral data]
{Higgs bundles, abelian gerbes\\ and cameral data}
\author[Oscar Garc{\'\i}a-Prada]{Oscar Garc{\'\i}a-Prada}
\address{Instituto de Ciencias Matem\'aticas \\ 
  Campus de Cantoblanco
  \\C/Nicol\'as Cabrera 13--15 \\ 28049 Madrid \\ Spain}
\email{oscar.garcia-prada@uam.es}
\author[Ana Pe{\'o}n-Nieto]{Ana Pe{\'o}n-Nieto}
\address{Universit\'e de Gen\`eve\\
Section de Math\'ematiques\\
Route de Drize 7\\ 1227 Carouge\\ Switzerland}
\email{ana.peon-nieto@unige.ch}
\thanks{
The work of the first  author was partially supported by the Spanish MINECO under ICMAT
Severo Ochoa project No. SEV-2015-0554, and under grant No. MTM2013-43963-P.
The second author was supported by the FPU grant from Ministerio de 
Educaci{\'o}n.
}

\date{16/02/2019}
\begin{abstract}
We study the Hitchin map for $G_{\R}$-Higgs bundles on a smooth curve, where $G_{\R}$ is a quasi-split real form of a complex reductive algebraic group $G$. By looking at the moduli stack of regular $G_{\R}$-Higgs bundles,
we prove
it induces a banded gerbe structure on a slightly larger stack, whose band is given by sheaves of tori. This characterization
yields a cocyclic description of the fibres of the corresponding Hitchin map by means of cameral data. According to this, fibres of the Hitchin map are categories of principal torus bundles on the cameral cover. The corresponding points inside the stack of $G$-Higgs bundles are contained in the substack of points fixed by an involution induced by the Cartan involution of $G_{\R}$.  We determine this substack of fixed points and prove that stable points are in correspondence with stable $G_\R$-Higgs bundles. 
\end{abstract}
\maketitle
\tableofcontents
\section{Introduction}

Higgs bundles over a compact Riemann surface for complex reductive Lie groups were introduced by Hitchin in \cite{SDE}, and have since been object of intensive study, due to the rich geometry of their moduli spaces. A particularly interesting aspect is the Hitchin integrable system. This is a fibration by complex Lagrangian tori over the so called Hitchin base \cite{Duke, SimpsonMod2, Scog, Schaub, DG, NgoLemme}. This process, called abelianization, has proven useful in the study of  the geometry of moduli spaces \cite{Duke}. It constituted moreover one of the hints that led to the link with  mirror symmetry \cite{HT,DP}. 

Higgs bundles for real Lie groups arise naturally via the non abelian Hodge correspondence, which establishes a homeomorphism of the moduli space of Higgs bundles with the moduli space of representations of the fundamental group.
Higgs bundles provide, as already ilustrated by Hitchin in  \cite{SDE} a very powerful tool  to study the geometry and topology of 
the moduli spaces of representations. In particular, this approach has been very useful in identifying and studying higher Teichm\"uller 
spaces (see \cite{OscarHigher} for a survey).  

It is interesting to note that Higgs bundles for real Lie groups generalize the complex group case (as a complex group is a real group with extra structure), but they also arise as fixed points of Higgs bundles for complex groups via involutions \cite{GRInvol}. We will consider both aspects in different sections of this article.

A Hitchin map  can be defined for Higgs bundles for real groups. The goal of this paper is to study  this map for 
 quasi-split real groups. Simple quasi-split real groups include split forms and groups whose Lie algebra is $\su(n,n)$, $\su(n,n+1)$, $\so(n,n+2)$, and $\mathfrak{e}_{6(2)}$.  It is well known that any real form is isomorphic via inner equivalence to a quasi-split group. From Cartan's point of view, the class of any outer automorphism of order two contains a quasi-split real form. It turns out that in this quasi-split situation the Hitchin fibration can be abelianized \cite{Tesis}.

The Hitchin map fibres the non compact moduli spaces of Higgs bundles onto an affine space, the Hitchin base. As a result, it allows to split the moduli space into non compact subspaces (sections isomorphic to the base \cite{Teich,HKR}) and compact dimensions (the fibres). Abelianization is certainly important from the gauge-theoretic point of view, as it transforms the objects of study into simpler ones. Here, abelianization should be understood in a broader sense. Even for quasi-split groups, in general,  the fibres will not be abelian varieties (exceptions to this are split groups and $\SU(p,p)$). This is due to the fact that they are contained inside the compactified Jacobians of non-smooth curves (the so called spectral and cameral covers) \cite{SimpsonMod2,Schaub,DG}.  However, generic points of the fibres do define abelian groups. In terms of the stack, the Hitchin fibration of the regular locus defines an abelian gerbe; in particular, the fibres are categories of tori over cameral covers.  In order to study them, we use the cameral techniques introduced by Donagi--Gaitsgory \cite{DG}, and follow Ng\^o's formulation \cite{NgoLemme}. The spectral data for some quasi-split real form has been formerly studied by different authors (see \cite{Emi, SUpp+1, SchapTesis, SchapUpp}).

Let $G_\R$ be a {real reductive Lie group}. Following Knapp  \cite[\S VII.2]{K}, by 
this we mean a tuple
$(G_\R,H_\R,\theta,\langle\,\cdot\,,\,\cdot\,\rangle)$, where $H_\R \subset G_\R$ is a maximal compact subgroup, 
$\theta\colon \g
\to \g$ is a Cartan involution and $\langle\,\cdot\,,\,\cdot\,\rangle$ is a
non-degenerate bilinear form on $\g_\R$, which is $\Ad(G_\R)$-
and $\theta$-invariant, satisfying natural compatibility conditions. We will also need the notion of a
real strongly reductive Lie group (see Definition \ref{def:strongly_red}). Let  $\theta$ be the Cartan involution, with associated  Cartan decomposition
\begin{displaymath}
\g_\R = \h_\R \oplus \m_\R
\end{displaymath}
where  $\h_\R$ is the Lie
algebra of $H_\R$. The group $H:=H_\R^\C$ acts on $\m:=\m_\R^\C$ through the  isotropy representation.

Let $X$ be a compact Riemann surface and $L$ be a holomorphic line bundle 
over $X$. An $L$-twisted $G_\R$-Higgs bundle on $X$ is a pair
$(E,\varphi)$, where $E$ is a holomorphic principal $H$-bundle
over $X$ and $\varphi$ is a holomorphic section of $E(\m)\otimes L$,  
where $E(\m)= E \times_{H}\m$ is the $\m$-bundle associated to 
$E$ via the isotropy representation.  
The section $\varphi$ is called the  Higgs field.  
When $L$ is the canonical line bundle $K$ of $X$ we obtain the  familiar
theory of $G_\R$-Higgs bundles. When $G_\R$ is compact the Higgs field is
identically zero and a $L$-twisted $G_\R$-Higgs bundle is simply a
principal $G$-bundle, where $G:=G_\R^\C$. When $G_\R$ is complex $G_\R=H$ and the isotropy
representation coincides with the adjoint representation of $G_\R$.
This is the situation originally considered by Hitchin in \cite{SDE,Duke},
for $L=K$. 

Let $\mc{M}(G_\R)$ the  moduli space of isomorphism classes of  polystable $L$-twisted  $G_\R$-Higgs bundles. By considering a basis of homogeneous polynomials  $p_1,\dots,p_a\in\C[\m]^H$ (where $a$ is the rank  of $G_\R$ and $\deg(p_i)=d_i$), one obtains the Hitchin map 
$$
h:\mc{M}(G_\R)\to \mc{A}_L(G_\R)
$$
defined by evaluating $p_1,\dots, p_a$  on the Higgs field (see Section \ref{sec:preliminaries} for a more intrinsic definition and Remark \ref{rk:usual_defi_Hitchin_map} for the relation between both). In the above
$ \mc{A}_L(G_\R)\cong\bigoplus_{i=1}^a H^0(X,L^{d_i})$ is the Hitchin base. This construction also yields a stacky version of the Hitchin map $[h]_L$ on the stack of $G_\R$-Higgs bundles  $\Higgs(G_\R)$. 

In this paper we study the morphism $[h]_L$ after imposing a regularity condition on the Higgs field. Namely, we assume that $\phi(x)$ has a maximal dimensional isotropy orbit for every $x\in X$. We find that the Hitchin map defines a gerbe structure once the automorphisms have been extended (yielding a stack $\Higgsrt(G_\R)$). This means that locally on the Hitchin base, the Hitchin fibration can be identified with the classifying stack for some group. Moreover, the existence of a section \cite{HKR} proves that this is globally the case.

We describe the fibres in terms of cameral techniques for quasi-split real groups, that is, we identify them with categories of tori over cameral covers. We also study the action of the natural algebraic involution induced by $\theta$ on the stack $\Higgsr(G)$, where $G=G_\R^\C$. We describe the fixed points, finding in Corollary \ref{cor:split}  that for split groups these are order two points (see also \cite{SchapTesis}). Moreover, we prove that stable points are precisely points of $\Higgsr(G_\R)$.

 In Section \ref{sec:preliminaries} we recall the preliminaries of Higgs bundles and the Hitchin map associated to real reductive algebraic groups. We put the theory  in the suitable context of stacks, and explain how the usual notions for complex groups arise naturally in this way (cf. Remark \ref{rk:complex_Higgs_are_Higgs}).
 
  Section \ref{section complex case} reminds abelianization as per  Donagi and Gaitsgory \cite{D93,Scog,DG}, which identify and open subset of the Hitchin system for complex groups with a gerbe over the Hitchin base, and the fibres  as categories of principal torus bundles over cameral covers (cf. Theorem \ref{thm DG cameral data}). We also explain how their language compares to Ng\^o's \cite{NgoLemme}.
  
  Section \ref{sec:local_Higgs} studies a local model for the Hitchin map for quasi-split real groups. In contrast with the complex group case, in order to obtain a gerbe (first step towards a description in terms of categories of principal bundles) we need to modify either the Hitchin base (by substituting it by the orbit space of the isotropy representation, see Section \ref{sec:GR_Higgs}) or the automorphism group (by enlarging it so the Hitchin base becomes an orbit space under the suitable extension of the isotropy action, see Section \ref{sec:NGR_Higgs}).  The result of the former is a gerbe over a non separated algebraic space (see Lemma \ref{lm gerbe 1}). Regarding the  latter, we study the  local stack  $\stack{\mr/N_G(G_\R)}$, where $\mr\subset\m$ are the regular points and $N_G(G_\R)$ is the normaliser of $G_\R$ inside its complexification $G$ (which exists by algebricity of $G_\R$). This substack of the local stack of $G$-Higgs bundles $\stack{\gr/G}$ is a minimal gerbe containing the local stack of $G_\R$-Higgs bundles $\stack{\m/H}$ (see Remark \ref{rk:essential_image}). It admits a natural involution whose fixed point substack contains $\stack{\m/H}$ (see Proposition \ref{prop local gerbe}). Once the gerby structure has been proven, we move on to characterise the band of the gerbe (that is, the automorphism group of objects in the fibre, see Remark \ref{rk:band}) in terms of schemes of tori, namely, we identify schemes of regular centralisers with torus schemes. Section \ref{sec:ss_local_band} studies the case of semisimple Higgs bundles, case in which the fibres are easily identified with categories of tori over \'etale local cameral covers. See Corollary \ref{cor:local_hitchin_fibres}. The following section extends these results to ramified  local cameral covers using the results of \cite{DG,NgoLemme}, which yields the main result of this section (Theorem \ref{thm Jtheta=Ttheta}).
  
  Section \ref{sec:global} studies the structure of the Hitchin map based on the local Hitchin map. We deduce that the stack of $G_\R$ and $N_G(G_\R)$-Higgs bundles is a gerbe over a suitable space (Theorem \ref{gerbe Higgs bundles}), which is moreover neutral when the degree of the bundle is even. The alternative description of the band from Theorem \ref{thm Jtheta=Ttheta} is used to obtain a cocyclic description of the fibres of the Hitchin map for $N_G(G_\R)$-Higgs bundles in terms of cameral data, that is, principal torus bundles over cameral covers (Theorem \ref{thm fibres hm for Gtheta}). Finally, we characterise the cameral data of fixed points under the natural involution on the stack associated with $G_\R$ in Theorem \ref{thm:cam_data_fixed_pts} and prove that stable fixed points are precisely stable $G_\R$-Higgs bundles (cf. Proposition \ref{prop:stable_fixed_points}). 
  
Two appendices gather the most relevant elements on Lie theory (Appendix \ref{sec:Lie_th}) and the geometry of regular centralisers (Appendix \ref{sec:reg_central}). Some of the results therein are original and appeared in \cite{HKR,Tesis}. 

We finish this introduction with some remarks concerning higher dimensional schemes, analytic categories, a more abstract approach \`a la Donagi-Gaitsgory, and the non abelianizable case.

The formalism developed in this article can be applied to higher dimensional 
schemes. However, as explained in \cite{DG}, for dimension higher than $2$ there may not exist any regular
Higgs fields, as the subset of a scheme $S$ on which a Higgs field is not regular is expected to have codimension $3$.

Moreover, we work with the \'etale topology. However, most of the results explained are valid in the setting of the analytic site, as long as $X$ is smooth compact and projective.

Regarding the parallelism with Donagi and Gaitsgory's work,  the notion of abstract $G_\R$-Higgs bundles also makes sense in this setting. Such an object is given by a pair $(E,\sigma)$ where $E$ is an $H$-principal bundle and $\sigma:E\to\HN$ is an $H$-equivariant section to a scheme $\HN$ parametrising the centralisers of elements in  $\mr$ (see Appendix \ref{sec:reg_central} for details).
Here $\la$ is a maximal abelian subalgebra of $\m$.
 One can work out the whole picture in this setting, and obtain that abstract $N_G(G_\R)$-Higgs bundles are a gerbe over the appropriate substack of cameral covers. Most details can be found in \cite{Tesis}; it would be interesting to give an intrinsic description of the precise substack of cameral covers,  in the spirit of Remark \ref{rk:cam_cov+inv}. 

In this article we restrict attention to quasi-split real groups. We know by \cite[Theorem 4.3.13]{Tesis}, that the Hitchin fibration  for any other real form 
can not be   abelianized. This is a consequence of regular centralisers not being abelian for any other real groups. See  \cite[\S 4.3.1.2]{Tesis} for more detailed comments on the existence of cameral data in other cases. This 
phenomenon was observed for some non-split  classical real  groups in 
\cite{Emi,HSchap}.

\section*{Acknowledments} The authors thank L. \'Alvarez-C\'onsul, F. Beck, R. Donagi, J. Heinloth, C. Pauly, C. Simpson and especially, T. Pantev, for comments and discussions on the topics of this paper.
\section{Higgs bundles and the Hitchin map}\label{sec:preliminaries}
Let $G_\R$ be a real algebraic group. We will say it is {\bf strongly reductive} if it satisfies Knapp's definition of reductivity \cite[\S VII.2]{K}. See Definition \ref{def:strongly_red}. Namely, the group comes endowed with extra data $(G_\R, H_\R,\theta, \langle \,\cdot \,,\,\cdot \,\rangle)$ where 
$H_{\R}\leq G_{\R}$ is a maximal compact subgroup, $\theta$ is the Cartan involution on the Lie algebra $\g_\R$ of $G_\R$, and $ \langle \,\cdot \,,\,\cdot \,\rangle$ is a $\theta$-invariant non degenerate bilinear form on $\g_\R$. These data satisfy conditions $i)-v)$ on \cite[page 446]{K}. Consider the corresponding Cartan decomposition of the Lie algebra 
\begin{equation}\label{eq Cartan dec}
\g_{\R}=\h_{\R}\oplus\m_{\R}.
\end{equation}
Then, $\theta$ is identically $+1$ on $\h_{\R}$ and $-1$ on $\m_{\R}$. Restriction of the adjoint action of $\Ad_{\g_\R}:G_\R\to\Aut(\g_\R)$ induces an action of $H_\R$ on $\m_\R$ whose complexification $\iota: H\to\Aut(\m)$ is called the \textbf{isotropy representation}. Here $H:=H_\R^\C$, $\m:=\m_\R^\C$ denote the complexifications and likewise for any other real group or (vector subspace of a) Lie algebra. 

Note that $\g_\R=(\g_\R\otimes \C)^\sigma$ for $\sigma$ the involution given by $(X,Y)\mapsto (-Y,X)$ on $\g_\R\otimes\C\cong_\R\g\oplus\g$. Given a $\sigma$-invariant complex vector subspace $V\subset\g$, we will denote by $V_\R:=V\cap\g_\R=V^\sigma$.

\begin{rk}\label{rk:complex_red_is_red}
Particular examples of strongly reductive real Lie groups include: complex connected reductive Lie groups $G$, or their real forms $G_\R<G$. The latter are fixed point subgroups under an antiholomorphic involution $\sigma$. One can prove the existence of a compact subgroup $U\leq G$ with defining involution 
$\tau$ such that  $\sigma\tau$ is a holomorphic involution on $G$ whose differential restricts to $\theta$ on $\g_{\R}$. By abuse of notation, we denote $\theta:=\sigma\tau$. With this, $H=G^\theta$ by  \cite[Proposition 7.21]{K}.

Note also that complex reductive groups $G$ appear as real forms of $G\times G$ with associated holomorphic involution $(g,h)\mapsto (h,g)$. 

According to this, a complex reductive Lie group can be seen as a particular case of a real strongly reductive group or of a real form.
\end{rk}

Given a smooth complex projective curve
$X$, we denote by $\Xet$ the small \'etale site of $X$, and $\XET$ the big \'etale site. Recall \cite{Vistoli} that a site is a category endowed with a notion of covering satisfying some axioms. In our case, the
category underlying $\Xet$ has as objects \'etale morphisms 
$S\to X$, and as arrows, morphisms over $X$. As for $\XET$, the underlying category are schemes over $X$, with arrows morphisms over $X$. In both cases, the coverings of the Grothendieck topology are collections of \'etale morphismos $\{U_i\to S\}_{i\in I}$ for $S\in\Xet/\XET$ which are jointly surjective, that is, such that $\bigsqcup_{i\in I}U_i\to S$ is surjective. 

Now, let $L\to X$ be an \'etale line bundle, that we will assume to be of degree at least $2g-2$, and if $\deg L=2g-2$ then $L$ is assumed to be the canonical bundle $K$ of $X$.

Note that $L$ is also locally trivial in the analytic topology, 
by Oka's Lemma and 
\cite[Theorem 1.9]{Steinberg}. 
\begin{defi}
	An $L$-twisted $G_{\R}$-Higgs bundle on $X$ is a pair $(E,\phi)$ where $E\to X$ is an \'etale
	principal 
	$H$-bundle and $\phi\in H^0(X,E(\m)\otimes L)$, where $E(\m)$ denotes the bundle associated to $E$ via the isotropy representation. The section $\phi$
	is called the \textbf{Higgs field}.
\end{defi}
We can reformulate the above as follows: on the sites $\Xet$ or $\XET$, we can consider the transformation stack $\Higgs(G_{\R}):=\stack{\m\otimes L/H}$. Recall that for a scheme $Y$ endowed with an action of an algebraic group $F$ the quotient stack $\stack{Y/F}$ parametrises principal $F$-bundles $P$ together with $F$-equivariant maps $P\to Y$. Thus $\Higgs(G_{\R})$
is the stack that to each scheme $f:S\to X$ associates the category of $f^*L$-twisted $G_{\R}$-Higgs bundles over $S$.
We have then the equivalent
\begin{defi}
	An $L$-twisted $G_{\R}$-Higgs bundle on $X$ is a section 
	$$\stack{(P,\phi)}: X\to\Higgs(G_{\R}).$$
\end{defi}
\begin{rk}\label{rk:moduli_stack}
	Strictly speaking, if we are interested in a moduli problem over $X$, we should instead consider
	the pushforward $s_*\Higgs(G_\R)$ by the structural morphism $s:X\to\mathrm{Spec}(\C)$; the sections of
	this stack over a complex scheme $S$ are Higgs bundles 
	over $X\times_\C S$. 
\end{rk}
Let $\la\subset \m$ be a maximal abelian subalgebra, and $W(\la)$ the restricted Weyl group. This can be defined as $W(\la)=N_H(\la)/Z_H(\la)$ where $N_H(\la)$ denotes the normalizer of $\la$ in $H$ and $Z_H(\la)$ its centraliser. The well known Chevalley restriction theorem proves that
$\C[\m]^H\cong\C[\la]^{W(\la)}$. In other words, the GIT quotient $\m\sslash H$ is isomorphic to the quotient $\la/W(\la)$. Thus, there exists a morphism, called the Chevalley morphism
\begin{equation}\label{eq Chevalley r}
\chi:\m\to\la/W(\la). 
\end{equation}

The map (\ref{eq Chevalley r})  is $\C^\times$-equivariant for the $\C^\times$ action on $\m$ as a vector space and the one on $\la/W(\la)$ induced 
by the action on the graded ring $\C[\la]^{W(\la)}$. Recall that the latter is a ring generated by homogeneous polynomials $p_i$ $i=1,\dots,a:=\dim\la$ of fixed degrees $d_i=\deg p_i$ which are invariants of the group. The action of $\C^\times$ on $\la/W(\la)$ is totally determined by $t\cdot p_i=t^{d_i}p_i$ where the RHS is the usual multiplication. Hence, it induces a morphism 
\begin{equation}\label{eq:chev_mL}
\m\otimes L\to(\la\otimes L)/W(\la).
\end{equation}
Moreover, \eqref{eq:chev_mL} is $H$-invariant, so we obtain 
$$
\shit_L:\stack{\m\otimes L/H}\to(\la\otimes L)/W(\la).
$$
\begin{defi} Let $\mc{A}(G_{\R})=tot(\la\otimes L/W(\la))$ be the total space of the bundle $\la\otimes L/W(\la)$. 
	The map 
	\begin{equation}\label{eq hitchin real}
	\shit_L: \stack{\m\otimes L/H}\to\la\otimes L/W(\la)  
	\end{equation}
	is called the \textbf{$L$-twisted Hitchin map} associated to $G_{\R}$. The scheme
	$\mc{A}(G_{\R})$ is called the \textbf{Hitchin base scheme} 
	(associated to $G_{\R}$).
\end{defi}
\begin{rk}\label{rk:complex_Higgs_are_Higgs}
	From Remark \ref{rk:complex_red_is_red} we may consider $\Higgs(G)$ for a complex reductive Lie group $G$. In these terms we recover the usual definition, as $\h=\m=\g$. In that case we have that $\la$ is a Cartan subalgebra of $\g$, $W(\la)$ is the associated Weyl group, and hence $\shit_L$ is the usual Hitchin map. 
\end{rk}
\begin{rk}\label{rk:usual_defi_Hitchin_map}
	Note that the choice of a basis of homogeneous generators $p_1,\dots, p_a$, $a=\mathrm{rk}\, G_\R$ of $\C[\m]^H$ induces an isomorphism $H^0(X, \la\otimes L/W(\la))\cong \bigoplus_{i=1}^aH^0(X,L^{d_i})$, where $d_i=\deg(p_i)$.
\end{rk}
\begin{rk} In relation to Remark \ref{rk:moduli_stack}, the condition on the degree of $L$ ensures representability of the Hitchin base scheme  $s_*\laC\otimes L/W$ by $H^0(X,\laC\otimes L/W(\la))$. Namely, for any complex scheme $S$, $S$-points of $s_*\laC\otimes L/W$ are simply $H^0(X,\laC\otimes L/W(\la))$.
\end{rk}
Note that the Hitchin map \eqref{eq hitchin real} can be defined in more generality without fixing the line bundle $L$. Following Ng\^o  \cite[\S2]{NgoLemme} we can consider the stack 
$
\stack{\m/H\times \C^\times}
$, which parametrizes pairs $(E\times P,\psi)$ where $E$ is a principal $H$ bundle, $P$ is a line bundle and $\psi:E\times P\to \m$ is an $H\times\C^\times$ equivariant morphism. The latter is equivalent to having an $H$-equivariant morphism  $\psi:E\to \m\otimes P$, so that by considering $\stack{\m/H\times \C^\times}$ we are parametrising all twists at once.

The same arguments as before imply that the Chevalley morphism induces 
\begin{equation}
\label{eq chevalley all line bundles}	
\stack{\tilde{\chi}}:\stack{\m/H\times \C^\times}\to\stack{\la/W(\la)/\C^\times}.
\end{equation}
Furthermore, by mapping each of the above stacks to $B\C^\times$ via the respective forgetful morphisms, 
we obtain
a commutative diagram:
$$
\xymatrix{
	\stack{\m/H\times \C^\times}\ar[rr]\ar[dr]&&\stack{\la/W(\la)/\C^\times}\ar[dl]\\
	&B\C^\times\times\la/W(\la)&
}
$$
where all the stacks above are seen as sheaves over the (small) \'etale site of $\la/W(\la)$.

Fixing a line bundle on $X$ is equivalent to considering a map $\stack{L}:X\to B\C^\times$; so one recovers $\shit$ by 
looking at the restriction of $\stack{\tilde{\chi}}$ to $\stack{L}:X\to B\C^\times$. This implies in particular that in order to study the Hitchin map (\ref{eq hitchin real}), we can restrict attention to the easier map (\ref{eq chevalley all line bundles}).
The advantage to this approach is that geometric questions reduce to Lie theoretic questions in this context, which are 
much easier to handle. Moreover, local triviality of bundles implies that at a local scale, the stack is essentially
the quotient $\stack{\m/H}$. So we will first study the local case before coming back to global Higgs bundles.
 
\section{The complex group case}\label{section complex case}
Before we turn to the general case, we recall in this section case of complex groups, which was studied by Donagi and Gaitsgory \cite{DG}. We also explain the reformulation by Ng\^o \cite{ NgoLemme}, and how both relate.

Let $G$ be a complex reductive algebraic group, and let $\g$ be its Lie algebra. Let $X$, $\Xet$, $\XET$ and $L$ be as in Section \ref{sec:preliminaries}.

 As explained there, the classifying stack of $L$-twisted $G$-Higgs bundles over $X$ is 
the quotient stack $\Higgs(G)=\left[\g\otimes L/G\right]$ over $\Xet/\XET$. 

As explained in Remark \ref{rk:complex_Higgs_are_Higgs},  the Chevalley morphism \eqref{eq Chevalley r} induces the Hitchin map
\begin{equation}\label{eq sHitchin}
 \shitc:\Higgs(G)\to\lt\otimes L/W(\lt)
\end{equation}
where $\lt\subset \g$ is a  Cartan subalgebra  with associated Weyl group $W(\lt)$.

Let $\gr\subset\g$ be the subset of regular elements (i.e., elements with maximal dimensional adjoint orbits). We can consider
the open substack 
$$\Higgsr(G)\subset\Higgs(G),$$
consisting of Higgs bundles whose Higgs field is everywhere regular, that is, its pointwise centraliser is minimal dimensional (of dimension $\mathrm{rk}\,G$). 

A more general object, whose relation to $\Higgsr(G)$ we explore in what follows, was studied by Donagi and Gaitsgory \cite{DG}. They consider the stack $\aHiggs(G)$ of \textbf{abstract Higgs bundles}. To explain this,  consider 
$\GN\subset Gr(r,\g)$ the smooth
scheme parameterising regular centralisers, namely, its points are centralisers of regular elements of $\g$, cf. \cite[Proposition 1.3]{DG}.  In the notation, $N:=N_{G}(T)$ is the normaliser of the maximal torus  $T$  
corresponding to $\lt$. The scheme $\GN$ is a partial compactification of the quotient $G/N$ (which parameterises Cartan
subalgebras). Then $\aHiggs(G)$
parameterises pairs $(E,\sigma)$ of a 
principal $G$-bundle $E\to X$ and a $G$-equivariant map $\sigma: E\to \GN$. With this, $\aHiggs(G)\cong\left[\GN/G\right]$.

 Note that there is a morphism of stacks
\begin{equation}\label{eq:morph_pairs_abstract}
\mathcal{C}: \Higgsr(G)\to\aHiggs(G)
\end{equation}
sending a pair $(E,\phi)\mapsto (E,\sigma_\phi)$, where $\sigma_\phi: E\to\GN$ associates to $p\mapsto \lc_{\g}(\phi(p))$.  Here,
we consider the Higgs field as a $G$-equivariant map $\phi:E\to \g\otimes L$. The map $\sigma_\phi$ is well defined, as for any $\lambda\in\C^\times$ $\lc_{\g}(\lambda\phi(p))=\lc_{\g}(\phi(p))$.

\begin{rk}In fact, the morphism $\mathcal{C}$ is well defined for Higgs bundles in $\Higgs(G)$ with generically regular Higgs field (cf. \cite[\S17.4]{DG}), but for the purposes of this paper we will stick to everywhere regular objects.
\end{rk}
A \textbf{cameral cover} of $X$ is a $W$-covering $\spec\to X$ which is \'etale locally a pullback of $\lt\to\lt/W$. Note that to
any $L$-twisted Higgs bundle $(E,\phi)$, it corresponds a section $b(E,\phi)\in H^0(X,\lt\otimes L/W)$ obtained by composing
$$
b(E,\phi)=\shitc\circ [(E,\phi)]:X\to\lt\otimes L/ W.
$$
where $(E,\phi)$ is seen as a morphism
$$
[(E,\phi)]:X\to\Higgsr(G).
$$

In this way, we assign to
$(E,\phi)$ the cameral cover defined by the Cartesian diagram
\begin{equation}\label{eq complex cameral}
 \xymatrix{
 \widehat{X}_b\ar[r]\ar[d]&{\lt}\otimes L\ar[d]_{\pi}\\\
 X\ar[r]_-{b(E,\phi)}&{\lt}\otimes L/W.
 }
\end{equation}
An abstract Higgs bundle $[(E,\sigma)]:X\to\aHiggs(G)$ also induces a W-Galois ramified cover. 
To see this, one defines a  ramified 
$W$-Galois cover $\GT\to\GN$, where $\GT$ is the incidence variety inside $\GN\times G/B$ for a given Borel subgroup $B$ containing $T$ (see \cite[Proposition 1.5]{DG}). 
An {\bf abstract cameral cover} is a $W$-Galois ramified cover locally isomorphic to $\GT\to\GN$.  These are classified by the stack $\Cov$  of { abstract cameral covers}. 

We may define an  \textbf{abstract Hitchin map}
\begin{equation}\label{eq abst shit}
 \ashit:\aHiggs(G)\to \Cov
\end{equation}
as follows. Any abstract Higgs bundle $(E,\sigma)$ induces the abstract cameral cover 
$\hat{E}=E\times_{\GN}\GT\to E$. A descent argument allows to prove that it comes from a unique $W$-cover $\specc\to X$, the {abstract cameral cover} associated to $(E,\sigma)$. 

Note  here that there is a morphism
\begin{equation}\label{eq:cam_cov_are_abs_cam_cov}
	\mc{C}':\mc{A}_L(G)\longrightarrow\mc{C}ov
	\end{equation}
	due to the following fact (cf. \cite[Proposition 1.5]{DG}).
\begin{fact}\label{fact:springer_res}
Consider the Grothendieck--Springer resolution $\tgr$ of $\gr$ (note that the Grothendieck--Springer resolution restricts to the regular locus), i.e.,  the $W$-cover obtained by
pullback 
\begin{equation}\label{eq GS resolution}
 \xymatrix{
 \tgr\ar[r]\ar[d]&{\lt}\ar[d]\\
 \gr\ar[r]&{\lt}/W.
 }
\end{equation}
Let $c:\gr\to\GN$ be the map sending  each element to its centraliser. Then  \cite[Proposition 10.3]{DG}
$$
\gr\times_{\GN}\GT\cong \tgr
$$
over $\gr$.
\end{fact}
Moreover, $\stack{h_{abs}}\circ\mc{C}=\mc{C}'\circ\shit_L$. So we will henceforth refer to both kinds of covers (namely, locally isomorphic to $\GT\to\GN$ or to $\lt\to\lt/W$) as cameral covers.
\begin{rk}\label{rk:cam_cover_im_C}
	Note that over the image of $\mc{C}$, the construction of a cameral cover is easier. Indeed, let $(E,\sigma)=\mc{C}(E,\phi)$. Then, the datum of $\sigma$ is equivalent to a section $\lc(\phi)$ of $E(\GN)$. The cameral cover $\hat{X}$ is the fibreed product  $X\times_{E(\GN)}{E}(\GT)$.
\end{rk}
Over $\mc{C}ov$ one may consider the following sheaf of groups:
\begin{equation}\label{eq:mT}
\mc{T}({S})=\left\{s:\widehat{S}\to T\ :\ \begin{array}{l}
	s(u\cdot w)=s(u)^w\  \pt w\in W,\\
	\alpha(s(x))\neq -1\ \pt u\in\wh{S}: s_\alpha(u)=u
\end{array}
\right\}
\end{equation}
where $\hat{S}$ is the cameral cover associated to $S\to \mc{C}ov$ and 
$s_\alpha\in W$ denotes the reflection with respect to the root $\alpha$.
 This sheaf of groups is representable by a group scheme $\mathbf{T}$, namely, $\mc{T}({S})=\Hom(S,\mathbf{T})$ are the $S$-points of $\mathbf{T}$. Note that its restriction to $\mc{A}_L(G)$ has $S$-points for $b:S\to\lt\otimes L/W$ given by
\begin{equation}\label{eq:mTL}
	\mc{T}_L(S):=\mc{C'}^*\mc{T}=\left\{s:\widehat{S}_b\to T\ :\ \begin{array}{l}
		s(u\cdot w)=s(u)^w\  \pt w\in W,\\
		\alpha(s(x))\neq -1\ \pt u\in\wh{S}: s_\alpha(u)=u
	\end{array}
	\right\}.
\end{equation}
Before we can state the first main result concerning the structure of the Hitchin map, let us include some preliminary definition.
\begin{defi}\label{def:gerbe}
Let $A\to Y$ be an abelian group scheme. An $A$-banded gerbe over $Y$ is a stack $\mathcal{X}$ which is locally (in a chosen  Grothendieck topology) isomorphic to the classifying stack $BA$. The group scheme $A$ is called the band of the gerbe.
\end{defi}
\begin{rk}\label{rk:band} We warn in here that our notion of the band differs in principle from that of Giraud's. Indeed, given a $\mc{G}$-gerbe $\mc{X}$, the band as defined by Giraud \cite{Giraud} is an $\Out(\mc{G})$-torsor $\mc{K}$. When $\mc{G}$ is abelian, however, $\mc{K}$
	can be lifted to an $\Aut(\mc{G})$-torsor $\ol{\mc{K}}$. In Definition \ref{def:gerbe},  the ``band'' rather means  $\ol{\mc{K}}\times_{\Aut(\mc{G})}\mc{G}$, which is locally isomorphic to $\mc{G}$. 
\end{rk}
\begin{thm}[Donagi--Gaitsgory, Ng\^o]\label{thm:T_gerbes_DG}
 The Hitchin map (\ref{eq sHitchin}) and the abstract Hitchin map (\ref{eq abst shit}) induce an
 abelian banded gerbe structure on $\Higgsr(G)$ and $\aHiggs(G)$ respectively. 
 The respective bands are isomorphic to the  group schemes  $\mc{T}_L$ and $\mc{T}$.
\end{thm}

 A major result in \cite{DG} is to give a cocyclic interpretation of the stack $\Higgsr(G)$ in terms of principal bundles over cameral covers, the so called \textbf{cameral data}.  
 
 \begin{defi}\label{def:mcCD}
 Denote by $\mc{C}am$ (for cameral data) the stack over $\mc{A}_L(G)$ that assigns to each 
 $b\in H^0(S,\lt\otimes f^*L/W)$  (for $f:S\to X$) 
 the category of $R$-twisted, $N$-shifted $W$-equivariant principal $T$-bundles on $\widehat{S}_b$, the cameral cover associated with $b$ (see \cite[\S 6.2]{DG}). 
 This category has as objects triples $(P,\gamma,\beta)$ where
 
 1. $P$ is a principal $T$-bundle on $\widehat{S}_b$.
 
 2. A map $\gamma$ fitting in the map of short exact sequences
 $$
 \xymatrix{
 	0\ar[r]&T\ar[r]\ar[d]&N\ar[r]\ar[d]_\gamma&W\ar[r]\ar[d]_{id}&1\\
 	0\ar[r]&\Hom(\widehat{S}_b,{T})\ar[r]&\Aut_W(P)\ar[r]&W\ar[r]&1.
 }
 $$
 In the above,  $\gamma$ assigns to each $n_w\in N$ inducing $w\in W$ an isomorphism
 $$\gamma(n):P\cong w(P)\otimes \mc{R}_w$$  (that is, an object of $\Aut_W(P)$). Here,
 $$w(P)=w^*P\times_w P
 $$
  and $\mc{R}_w$ is the principal $T$-bundle associated to the ramification locus of $w$. When $w=s_\alpha$ is the reflection associated with $\alpha$, $\mc{R}_w=\check{\alpha}(\mc{O}(D_\alpha))$, where $D_\alpha\subset\specu_b$ are the fixed points under $s_\alpha$. This defines a cocycle in $Z^1(W,BT)$, so that the $\mc{R}_{s_\alpha}$ for simple roots determine $\mc{R}_w$ for all $w\in W$. See \cite[\S 5]{DG} for details.
 
 3. $\beta=(\beta_i)_{i=1}^r$ is a family of isomorphisms 
 $$\beta_i: \alpha_i(P)|_{D_{\alpha_i}}\cong \mc{O}(D_{\alpha_i})|_{D_{\alpha_i}},$$
 for each simple root ${\alpha_i}$, with $D_{\alpha_i}$ the corresponding ramification divisor.  
 
 The above data are subject to compatibility conditions for which we refer the reader to  \cite[\S 6.2]{DG}. 
 \end{defi}

\begin{thm}\label{thm DG cameral data}
The stacks $\mc{C}am$ and $\Higgsr(G)$ are isomorphic. In particular, the fibre of the Hitchin map over $b\in \mc{A}_L(G)$ is the category of $R$-twisted, $N$-shifted $W$-equivariant principal $T$-bundles on $\widehat{X}_b$.
\end{thm}
\begin{proof}[Sketch of proof.]
Since we need to refer to the proof in the sequel, we will recall the main elements. For details, we refer the reader to \cite[Theorem 6.4]{DG}.

Firstly, we have a universal cameral datum 
\begin{equation}\label{eq:universal_cameral_datum}
(\Puniv,\guniv,\buniv)\to \GT,
\end{equation}
given by:

a)  The bundle $\Puniv:=\GT\times_{G/B}G/U$ is
given by the pullback of the principal $T$ bundle $G/U\to G/B$ onto the incidence variety $\GT\subset\GN\times G/B$. 

b) The datum $\guniv$ is obtained by observing that for reflections under simple roots $s_\alpha$, the isomorphism  $\Puniv|_{G/T}\cong s_\alpha^*\Puniv|_{G/T}$ extends meromorphically with associated divisor $-\check{\alpha}(D_{\alpha})$.

c) The datum $\buniv$ is checked from the above by reducing the statement to suitable Levi subgroups.

Given a Higgs bundle $(E,\phi)$, let $\mc{C}(E,\phi)=(E,\sigma_\phi)$ where $\mc{C}$ is as in \eqref{eq:morph_pairs_abstract}. Since $\sigma_\phi$ is $G$-equivariant
$$
E\stackrel{\sigma_\phi}{\longrightarrow}{\GN},
$$
then of $\hat{E}=E\times_{\GN}\GT\longrightarrow E$ is also $G$-equivariant, and so descends to a unique
$$
\hat{X}\longrightarrow	X.
$$
The same applies to $\guniv,\buniv$, which yields a cameral datum $(P_E,\gamma_E,\beta_E)$  on $\hat{X}$. Since the assignment $(E,\phi)\mapsto (P_E,\gamma_E,\beta_E)$ is functorial, it defines a morphism 
$$
\Higgsr(G)\longrightarrow\mc{C}am.
$$
By Theorem \ref{thm:T_gerbes_DG} and the discussion on page  38 of \cite{DG}), this is a morphism of of $\mc{T}_L$-gerbes, which must therefore be an isomorphism.
\end{proof}
\begin{rk}
Note that according to \cite[Corollary 17.6]{DG}, an element of $\Higgsr(G)$ is equivalent to an element of $\mc{H}iggs$ together with a $W$-equivariant embedding  $\hat{X}\to \lt\otimes L$ of schemes over $X$. The way we have defined cameral covers for $L$-valued Higgs bundles directly determines this embedding. 
\end{rk}

\section{The local situation: untwisted $G_\R$-Higgs bundles} \label{sec:local_Higgs}
	\subsection{Quasi-split real groups and the local Hitchin map}
	
The purpose of this section is to analyse the local Hitchin map 
\begin{equation}\label{eq lshit}
\lshit:\stack{\m/H}\to\la/W(\la), 
\end{equation}
constructed by $H$-equivariance of the Chevalley morphism \eqref{eq Chevalley r}. 

From now on we will assume that $G_{\R}$ is a quasi-split strongly  reductive algebraic group. Amongst simple groups, quasi-split forms include split groups and those whose Lie algebra is $\su(n,n)$, $\su(n,n+1)$, $\so(n,n+2)$ and $\mathfrak{e}_{6(2)}$.

We note in here that algebraic groups admit a complexification. We let $G:=G_\R$, and assume it to be connected.

 Recall that a real group is quasi-split if any of the following equivalent conditions hold:

(QS1) The centraliser $\lc_\g(\la)$ of $\la$ inside $\h$ is abelian. In this case $\lc_\g(\la)$ is a Cartan subalgebra.

(QS2) There exists a $\theta$ invariant Borel subgroup $B<G$ such that $B^\theta=B^{op}$ is the opposed Borel subgroup. Borel subgroups satisfying this  condition are  called $\theta$-anisotropic Borel subgroups. 


We fix once and for all an anisotropic Borel subgroup $B<G$, a $\theta$ invariant maximal torus $T<B$ which we assume to be the complexification of $T_\R=T\cap G_\R$. Assume that the Lie algebra of $T$  satisfies $\lt=\ld\oplus\la$ with  $\la=\lt\cap\m$ maximal and $\ld=\lt\cap\h$. We let $S\subset \Delta=\Delta(\g,\lt)$ be the associated sets of (simple) roots, $W$ the Weyl group. Similarly, we can define $\Sigma(\la)$ the set of restricted roots associated to $\la$ (which is the image of the restriction map $res|_{\la}:\Delta\to\la^*$). This is a root system (possibly non reduced). Reflection with respect to simple such roots generates a group $W(\la)$ called the restricted Weyl group. This group is also isomorphic to
$N_H(\la)/C_H(\la)$ and $N_{H_\R}(\la_{\R})/C_{H_\R}(\la_{\R})$, where $\la_\R=\la\cap\g_\R$. Moreover, any root $\alpha$ can be expressed as $\alpha=\lambda+i\beta$ with $\lambda\in\la_\R^*, \beta\in \ld_\R^*$. A Borel subalgebra is $\theta$-anisotropic if and only if it induces an ordering like this. Moreover, quasi-split real forms have no purely imaginary roots (a necessary condition to admit a $\theta$-anisotropic Borel subgroup).

The following lemma tells us that up to isogeny, strongly reductive real group are real forms.

\begin{lm}\label{lm:quasisplit real forms}
	Let $G_\R$ be a strongly reductive real algebraic group and $G$ its complexification. Then there exists a real form  $G^\sigma$ of $G$ such that:
	\begin{enumerate}
		\item\label{it:isogeny_real_gps} $\g_\R=\g^\sigma$, namely, $G_\R$ and $G^\sigma$ are isogenous.
		\item\label{it:same_normal} $N_G(G_\R)=N_G(G^\sigma)$.
	\end{enumerate}
\end{lm}
\begin{proof}
	To see	\eqref{it:isogeny_real_gps}, note that $\g_\R$ is a real form of $\g$. Let $\sigma$ be the involution defining it, and let $\theta=\sigma\tau$ be the Cartan involution compatible with a compact form $\lu=\g^\tau$ of $\g$. Let $U=\exp(\lu)<G$ be the corresponding compact subgroup. Then by \cite[\S V.II]{K} $G\cong U\times i\lu$, and the involution on $G$ defined by
	$$
	U\times i\lu\ni (e^Y, X)\stackrel{\sigma}{\mapsto} (e^{\sigma Y}, \sigma X)
	$$ 
	defines a real form $G^\sigma$. By construction, its Lie algebra is that of $G_\R$.
	
	To prove \eqref{it:same_normal}, note that strong reductivity of $G$ implies thats $N_G(G_\R)\subset N_G(\g_\R)$ (as $G$ acts by inner automorphisms on $\g$ and so does $N_G(G_\R)$ on $\g_\R$), so equality holds. On the other hand, real forms of strongly reductive complex groups are strongly reductive by \cite[Proposition 3.6]{HKR}, so also $N_G(G^\sigma)=N_G(\g^\sigma)$. We may conclude from \eqref{it:isogeny_real_gps}.
\end{proof}

Returning to the study of \eqref{eq lshit}, one cannot expect to have any ``nice'' (gerby) structure of the local Hitchin map (\ref{eq lshit}) as a whole. 
The reason is that
the 
inertia stack, which is the stack that classifies automorphisms of objects, is far from being flat. This is a necessary condition in order to have a gerbe structure \cite[Appedix A]{AOV}.

Indeed, consider  the following group scheme on $\m$:
\begin{defi}\label{defi Ctheta}
	We let $C^\theta\to\m$ be the group scheme over $\m$ defined by
	$$
	C^\theta=\{(m,h)\in \m\times H\ |\ h\cdot m=m\}.
	$$ 
\end{defi}
\begin{rk}\label{rk:notation}
	The superscrip $^\theta$ in the notation is related with the fact that $H=G^\theta$ for a holomorphic involution $\theta$ on $G$ lifting the extension  of the Cartan involution by complex linearity (cf. Remark \ref{rk:complex_red_is_red}). Thus, if $C\subset \m\times G$ is the centraliser group scheme, $C^\theta$ is the fixed point set of $C$ by $(x,g)\mapsto (x,g^\theta)$.
\end{rk}
Note that there is an action of $H$ on $C^\theta$ (namely, the adjoint action) lifting the isotropy action on $\m$.
This means that the inertia stack of $\stack{\m/H}$ is induced from the latter sheaf. Indeed:
\begin{lm}\label{inertia}
	The sheaf represented by the group $C^\theta\to\m$ descends to the inertia stack on $\stack{\m/H}$.
\end{lm}
\begin{proof}
	This is standard: sections of $\stack{\m/H}$ are, up to covering, trivial bundles together with 
	equivariant maps to $\m$. Automorphisms of a pair $(P,\phi)$ over $S\in\CET$ are sections of the adjoint 
	bundle $\Ad(P)$ centralising $\phi$,  which locally are just equivariant maps to $C^\theta$.
\end{proof}
Consider  the
regular locus:
\begin{defi}\label{def:mr}
	We denote by 
	$$
	\mr=\{x\in\m\ :\ \dim C_{H}(x)=\dim\la\}.$$
	This is the set of regular elements, and it
	can be proved it corresponds to elements of $\m$ with minimal dimensional centraliser, or maximal dimensional 
	(isotropy) orbits \cite[Proposition 7]{KR71}.\end{defi}
\begin{rk}\label{rk mr c gr}
	An equivalent definition of quasi-splitness is to assume that  $\mr=\m\cap\gr$ where $\gr\subset\g$ is the subset of (adjoint) 
	regular elements. This implies in particular that $\lc_{\m}(x)$ has dimension equal to the dimension of $\la$, so in particular, if $x\in\la$ is regular, then $\lc_{\m}(x)=\la$.
\end{rk}
\begin{defi}
	We will call the stack $\stack{\mr/H}$ the stack of everywhere regular local $G_{\R}$-Higgs bundles. 
\end{defi}
\begin{lm}\label{C is smooth}
	$C^\theta\to \mr$ is smooth.
\end{lm}
\begin{proof}
	This proof is similar to the one of  \cite[Proposition 11.2]{DG}.
	
	Given a complex point 
	$(x,h)\in C^\theta(\C)$, we have that the tangent space $T_{(x,h)}C^\theta(\C)$ is defined inside 
	$T_{(x,h)}\mr\times  H=\m\times \h$ by the equation
	$d_{(x,h)} f(y,\xi)=0$ where
	$$f(x,h)=\Ad(h)(x)-x$$
	Now: 
	$$
	\frac{\partial}{\partial h}|_{(h,x)} f(y,\xi)=\frac{\partial}{\partial h}|_{(h,x)} \Ad(h)\circ ev_x(y,\xi)
	=h\cdot[\xi,x]
	$$
	Hence $d_{(x,h)}f(y,\xi)=\Ad(h)([\xi,x])+h\cdot y-y$. Clearly, the differential of the map 
	$C^\theta\to\mr$
	sends $(y,\xi)\mapsto y$. So all we need to check is that
	$$
	\{y\in\m\ |\ y-h^{-1}(y)\in\ad(x)(\h)\}=\m
	$$
	One inclusion is clear, so lets see that any $z\in\m$ satisfies the condition. First note that 
	$$
	\g\cong[x,\g]\oplus\lc_{\g}(x)\cong[x,\h]\oplus\lc_\m(x)\oplus[x,\m]\oplus\lc_\h(x)
	$$
	so that
	$$
	\m\cong[x,\h]\oplus\lc_\m(x)
	$$
	Since the action of any $h\in C_H(x)$ respects the direct sum, it is enough to check that 
	$$
	\lc_\m(x)\subseteq \{y\in\m\ |\ y-\Ad(h^{-1})(y)\in\ad(x)(\h)\}
	$$
	First, suppose that $C_H(x)$ is connected. Then, by quasi splitness, 
	$\lc_\h(x)=\lc_\h(\lc_\m(x))$ (as the centraliser of a regular element is abelian),
	so that for any $z\in \lc_\m(x)$, $\Ad h(z)-z=0\in\ad(x)(\h)$. 
	
	Now, for the non connected case, since fibres are algebraic groups in characteristic zero
	they are smooth. Thus independently of the component $h$ is in, the dimension of the tangent 
	bundle will not vary.
\end{proof}
In order to obtain a gerbe, we need to ``quotient'' the stack $\stack{\mr/H}$ by the inertia stack parameterizing
automorphisms of objects, thus obtaining a gerbe over the sheaf of isomorphism classes of objects. This process is called rigidification \cite{AOV}. In the complex 
group case, the sheaf of isomorphism classes is the GIT quotient $\g\sslash G$. 
The situation in the real group case is somewhat different, as $\left[\mr/H\right]$ fails to be locally connected over the
GIT quotient $\la/W(\la)\cong\m\sslash H$. This is due to the fact that $\la/W(\la)$ does not always parameterise $H$-orbits, but rather orbits of the larger group $N_G(G_\R)$. So in the forthcoming sections we analyse both possible gerby structures, first by rigidifying (Section \ref{sec:GR_Higgs}) and secondly by increasing the automorphism group (Section \ref{sec:NGR_Higgs}). While the first approach is valid for any strongly reductive group by Lemma \ref{C is smooth} and \cite[Theorem A.1]{AOV}, the second is totally determined by real forms (by Lemmas \ref{lm:quasisplit real forms} and \ref{lm:mGtheta_are_NGR}).

\subsection{The case of real forms}\label{sec:GR_Higgs}

When $G_\R$ is a real form,  the following example shows that we cannot expect for \eqref{eq lshit} to induce a gerbe structure. 
\begin{ex}\label{ex orbits sl2}
 Let $G_\R=\SL(2,\R)$. By suitably choosing the involutions, we may identify
 $\SO(2,\C)$ with diagonal matrices, and $\lie{sym}_0(2,\C)$ (the subspace of zero traced symmetric matrices) with off diagonal matrices. 
 
 The Chevalley morphism is $\det: \lie{sym}_0(2,\C)\to \C$. Nilpotent elements are precisely $\det^{-1}(0)$. Now, all elements
 of $\lie{sym}_0(2,\C)\setminus 0$ are regular. The action of $\SO(2,\C)$ on $\lie{sym}_0(2,\C)$ is given by:
 \begin{equation}\label{eq action SO2}
\lambda\cdot  \left(\begin{array}{cc}
        0&\beta\\
       \gamma&0
       \end{array}
\right)=\left(\begin{array}{cc}
        0&\lambda^2\beta\\
       \lambda^{-2}\gamma&0
       \end{array}
\right).  
 \end{equation}
So we see that for nilpotent elements there are two maximal open orbits (with $0$ in the closure of both): upper and lower triangular matrices. These however become 
one via conjugation by
$\left(\begin{array}{cc}
        0&i\\
       i&0
       \end{array}
\right)$, which together with $\SL(2,\R)$ generates the group $N:=N_{\SL(2,\C)}(\SL(2,\R))$,
 which thus fits into an exact sequence
$$
1\to \SL(2,\R)\to N\to\Z_2\to 1.
$$
Now, semisimple elements in $\lie{sym}_0(2,\C)$ are elements with both off diagonal entries different to
zero; hence by suitably choosing $\lambda\in \C^\times$ in (\ref{eq action SO2}), 
we see that for any such element $x$, $N\cdot x=\SO(2,\C)\cdot x$.
\end{ex}
Hence,  the Hitchin map does not  induce a gerbe structure, but we may replace the Hitchin base by the space of orbits as follows. 

Consider the sheaf
$\mathbf{A}$ associated to the following presheaf:
\begin{equation}\label{eq pi0m}
\mathbf{A}':Sch\longrightarrow Sets\qquad S\mapsto\{\textrm{isomorphism classes in }\stack{\mr/H}(S)\}.
\end{equation}
There is a surjective morphism $\stack{\mr/H}\to \mathbf{A}$. Moreover:
\begin{lm}\label{lm gerbe 1}
 $\stack{\mr/H}\to \mathbf{A}$ is a gerbe banded by the space fibreed in abelian groups $J^\theta\longrightarrow S$ obtained by faithfully flat descent from $C^\theta$.
\end{lm}
\begin{proof}
 This is just \cite[Theorem A.1]{AOV}, where the group stack by which one rigidifies it the full inertia stack, which applies by Lemma \ref{C is smooth}. Note this applies to any strongly reductive quasi-split real group.
\end{proof}
\begin{rk}\label{rk pi0}
 Note that $\mathbf{A}$ is the orbit space of the isotropy representation. In the particular case of $\SL(2,\R)$ this is
 the non separated scheme consisting of a line with a double origin. In general, it will be a non separated algebraic
 space.
\end{rk}
\subsection{The associated gerbe of $N_{G}(G_{\R})$-Higgs bundles}\label{sec:NGR_Higgs}
The Hitchin map \eqref{eq lshit} fails to be a gerbe even when restricted to the regular locus, but by enlarging the automorphisms
slightly we obtain a nicer structure.
\begin{defi}\label{defi Gtheta}
 Let $\theta$ be the holomorphic involution on $G$ associated to the real form $G_\R$, in such a way that $G^\theta=H$.  Consider the algebraic group
 $$
G_\theta=\{g\in G\ :\ g^{-1}g^\theta\in Z(G)\}.
 $$
 \end{defi}
 We will see that the stack $\stack{\mr/\Gtheta}$ is a gerbe under the local Hitchin map
 \begin{equation}\label{eq lshittheta}
 \lshit:\stack{\m/\Gtheta}\to\la/W(\la), 
 \end{equation}
 obtained by $\Gtheta$ equivariance of \eqref{eq Chevalley r}.
 \begin{lm}\label{lm:mGtheta_are_NGR}
 	Let $N_{G}(G_\R)$ be the normaliser inside $G$ of the real form $G_\R$. Then
 	$\stack{\mr/\Gtheta}$ is the stack of local $N_G(G_\R)$-Higgs bundles.
 \end{lm}
 \begin{proof}
 	Since 
 	$\mathfrak{n}_\g(\g_\R)=\g_\R+\z(\g)\cap {i\g_{\R}}$, all we need to check is that the maximal compact subgroup of $N_{G}(G_\R)$ complexifies to $\Gtheta$.
 	
 	Let $G_\R=G^\sigma$, and let $U<G$ be maximal compact subgroup defined by an involution $\tau$ commuting with $\sigma$. Then $\theta=\sigma\tau$ and
 	$$
 	N_G(G_\R)=\{g\in G\,:\, g^{-1}g^\sigma\in Z_G(G_\R)\}.
 	$$ 
 	Thus its maximal compact subgroup is
 	$$
 		N_U(G_\R)=\{g\in U\,:\, g^{-1}g^\sigma\in Z_G(G_\R)\}=\{g\in U\,:\, g^{-1}g^\theta\in Z_G(G_\R)\}.
 	$$
 	Hence, if we check that $Z_G(G_\R)=Z(G)$, the result follows, as
 	$$
 	\Gtheta\cap U=\{g\in U\,:\,  g^{-1}g^\theta\in Z(G)\}.
 	$$
 	First, $G$ is strongly reductive (reductive in the sense of \cite[\S VII.2]{K}), so by definition $G$ acts on its Lie algebra by inner automorphisms, and so $Z_G(G_\R)$ centralises $\g_\R$, and so also $\g$. But then it centralises $G^0=G$.
\end{proof}
  \begin{rk}\label{rk:Gtheta_real_bundles}
 In fact, we may obtain a gerbe in different ways from the local Hitchin map \eqref{eq lshit}. Independently of whether  $G$ is strongly reductive or not, the stack of everywhere regular $N_G(G_\R)$-Higgs bundles defines a gerbe over the same space as $\stack{\mr/\Gtheta}$, since the space of orbits in $\mr$ under $\Gtheta$ and $N_U(G_\R)^\C$ is the same in virtue of Propositions \ref{prop orbits in laW} and \cite[Proposition 3.21]{HKR}.
  \end{rk}
Now, on  $\mr$, we consider the following group scheme:
\begin{equation}\label{eq Ctheta}
 C_\theta=\{(m,h)\in \mr\times G_\theta\ |\ h\cdot m=m\}. 
\end{equation}
\begin{lm}
 The group scheme $C_\theta\to\mr$ is a smooth abelian group scheme.
\end{lm}
\begin{proof}
 Commutativity follows from quasi-splitness, as for any $x\in\gr$ $C_{G}(x)$ is abelian (cf. Remark \ref{rk mr c gr}) and $(C_\theta)_x\subset C_{G}(x)$.
 
 For smoothness of the morphism, it is enough to observe that $\g_\theta=\h\oplus\z_{\m}(\g)$, so the proof follows from the same argumenst used in Lemma \ref{C is smooth}
 
 
 

\end{proof}
\begin{lm}\label{lm exact seq Gtheta}
There is an exact sequence 
\begin{equation}\label{eq SES Gtheta}
1\to H\to \Gtheta\to F^2\to 0,\end{equation}
 where $F=\{a\in \exp (\la)\ :\ a^2\in Z(G)\}$
and the map $\Gtheta\to F^2$ is given by $g\mapsto g^{-1}g^\theta$.
 Moreover, $\Gtheta=FH$.
\end{lm}
\begin{proof}
	By Lemma \ref{lm KA_uK} and connectivity of $G$, its maximal compact subgroup $U$ has the form $HA_uH$. Now, let $g\in \Gtheta$, and let $g=ue^{iV}$ be its polar decomposition, with $u\in U$, $V\in \lie{u}$. By uniqueness of the latter and reductivity of
	$Z(G)$ (see \cite[Corollary 7.26]{K}), 
	$$
	g^{-1}g^\theta\in Z(G)\iff u^{-1}u^\theta\in Z_U(G), e^{-iV}e^{i\theta V}=e^{-i(V-\theta V)}\in e^{i\lie{z(u)}}.
	$$   
	But now 
	$u=h_1ah_2$ for $h_1$, $h_2\in H$, $a\in A_u$, so 
	$$u^{-1}u^\theta=\Ad(h_2^{-1})(a^{-2})\in Z_U(G)\iff 
	\Ad(h_2^{-1})(a^{-2})=a^2\in Z_U(G).
	$$ 
	On the other hand, given that $i\lu=i\h\oplus \m$, we have ${i(V-\theta V)}\in  \m$, so  ${i(V-\theta V)}= \Ad(h)(iX)$ for some $h\in H$, $X\in i\la$ by \cite[Lemma 7.29]{K}. So clearly, 
	$$e^{i(V-\theta V)}\in e^{i\lie{z(\lu)}}\iff e^{i(V-\theta V)}=e^{iX}\in ie^{\lie{z(\lu)\cap\m}}\subset e^{\la}
	$$
	where the last inclusion follows by maximality of
	$\la$. Since clearly $F\subset\Gtheta$, exactness of the sequence follows. 
	
	To see that $\Gtheta=FH$,  first note that $g=h_1ah_2 e^{iV}$ for $h_i\in H$, $a\in F$. So all we need to check is $e^{iV}\in FH$. Write $V=V_h+V_m$, $V_h\in \h$, $V_m\in i\m$. Since $V-\theta V=2 V_m\in i\la\cap\z(\lu)$, it follows that $e^{iV}=e^{iV_j}e^{iV_m}$, with $e^{iV_j}\in H,\ e^{iV_m}\in F$.
\end{proof}
\begin{cor}\label{lm exact seq Ctheta}
	The scheme $C_\theta$ fits into an exact sequence 
	$$
	1\to C^\theta\to C_\theta\stackrel{p}{\to} \mc{F}^2\to 0.
	$$
	where $\mc{F}^2\subset \mr\times F^2$ is the group scheme defined as follows:
\begin{enumerate}
	\item\label{mcFrs} $\mc{F}^2|_{\mrs}=\mrs\times F^2$.
	\item\label{mcFrn} $\mc{F}^2_x=\{y^2\,:\,y\in Z(C_{\Gtheta}(\Ad(h)(x_s)))\cap F\}$, where $x=x_s+x_n$ is the Jordan decomposition and $h\in H$ is such that $\Ad(h)(x_s)\in \la$. 
\end{enumerate}
\end{cor}
\begin{proof}
	From Lemma \ref{lm exact seq Gtheta}, we have an exact sequence
	$$
	0\to\mr\times H\to\mr\times \Gtheta\to \mr\times F^2.
	$$
	So we need to characterise the image.
	
	Statement \eqref{mcFrs} is clear when restricted to $\lar:=\la\cap\mr$. Indeed,  $C_\theta|_{\lar}=\lar\times T_\theta$, where $T_\theta=C_{G_\theta}(\la)$, so that $T_\theta\supset A\cap T_\theta$. Then, from Lemma \ref{lm exact seq Gtheta} we may conclude.
	
	Suppose that $x\in\mrs$. Then, by \cite[Theorem 1]{KR71}, there exists some $h\in H$ such that $x'=\Ad(h)(x)\in\laC$. Thus, given that $p$ is $H$-equivariant, and that $F^2\subset Z(G)$, we may conclude that \eqref{mcFrs} holds.
	
To see \eqref{mcFrn}, for any $x\in\m$, if $x=x_s+x_n$ is its Jordan decomposition, it follows that $x_s, x_n\in \m$ by \cite[Proposition 3]{KR71}. Also, $C_{\Gtheta}(x)=C_{\Gtheta}(x_s)\cap C_{\Gtheta}(x_n)$. Since $\g_\R$ is quasi-split,  $C_{\Gtheta}(x_n)\cap C_{\Gtheta}(x_s)$ 
is unipotent within $C_{\Gtheta}(x_s)$, so the intersection is the center of $C_{\Gtheta}(x_s)$ (since centralisers of nilpotent elements are unipotent). By the same arguments as used in the semisimple case, the image of this into $F^2$ is independent of the conjugacy class. 
	\end{proof}
	\begin{rk}
		Given $x=x_s+x_n\in\mr$ such that $x_s\in\la$, it follows that for all roots $\alpha$ such that $\alpha(x_s)=0$ $F^\alpha=F\cap \Ker(\alpha)\subset Z(C_\theta(x_s))$. The latter is in fact generated by
		$Z(G)$ and $F^\alpha$ for roots vanishing on $x_s$.
		
		Note also that the semisimple part of $F$ acts on $\mr$ by permuting different $H$ orbits within the 
		same $\Gtheta$ orbit. This happens only for non semisimple elements by \cite[Proposition 3]{KR71}.
	\end{rk}
\begin{prop}\label{prop Ctheta descends}
The scheme $C_\theta$ descends to an affine abelian group scheme $J_\theta$ over $\la/W(\la)$. 
 \end{prop}
\begin{proof}
	First of all, note that $\Gtheta$ normalises $\m$, which induces an action on $C_\theta$ making $C_\theta\to\mr$ $\Gtheta$-equivariant. From this point on we can adapt the proof of Lemma 2.1.1 in \cite{NgoLemme} to our context. 
	
	Since  for $x\in\mr$,  $(C_\theta)_x$ is abelian, hence,  we can define the fibre over $\chi(x)\in\la/W(\la)$ to be $(C_\theta)_x$ itself. Any other choice 
	will be canonically isomorphic over $\la/W(\la)$ by \cite[Theorem 11]{KR71} and commutativity of the centraliser.
	As for the sheaf itself, it can be defined by descent of $C_\theta$ along the flat morphism
	$\mr\to\la/W(\la)$.
	
	For $C_\theta$ to descend,  both pullbacks to $\mr\times_{\la/W(\la)}\mr$  must 
	be isomorphic. Consider both
	projections 
	$$
	p_1, p_2:\mr\times_{\la/W(\la)}\mr\to\mr,
	$$
	and let $C_i=p_i^*C_\theta$. Consider $f: \Gtheta\times\mr\to\mr\times_{\la/W(\la)}\mr$ given by 
	$(h,x)\mapsto (x,h\cdot x)$. We will proceed by proving that there exists an isomorphism
	$f^*C_1\cong f^*C_2$ over $\Gtheta\times \mr$, and then check it descends to an isomorphism over 
	$\mr\times_{\la/W(\la)}\mr$ (since the above map is smooth and therefore flat).
	
	Consider
	$$
	\xymatrix{
		F:f^*C_1\ar[r]&f^*C_2\\
		\left((m,h),g\right)\ar@{|->}[r]&\left((m,h),\Ad_hg\right).
	}
	$$
	It defines an isomorphism over $\Gtheta\times \mr$. To see whether $F$ descends to $\mr\times_{\la/W(\la)}\mr$, 
	we need to check that $F(m,h,g)$ depends only on $(m,h\cdot m, g)$ and not on the particular element
	$h\in\Gtheta$. To see that, it must happen that the pulbacks of $F$ to 
	$$
	S:=\left(\Gtheta\times \mr\right)\times_{\mr\times_{{\la/W(\la)}}\mr}\left({\Gtheta}\times \mr\right)
	$$
	by each of the projections $\pi_1,\ \pi_2:S\to \Gtheta\times\mr$  
	fit into the commutative square 
	$$
	\xymatrix{
		\pi_1^*f^*C_1\ar[r]\ar[d]&\pi_1^*f^*C_2\ar[d]\\
		\pi_2^*f^*C_1\ar[r]&\pi_2^*f^*C_2.
	}
	$$
	To do this, note that $S\cong {\Gtheta}\times C_1$ by the map
	$$
	\left((m,h),(m,h')\right)\mapsto (h,(m,hm), h^{-1}h').
	$$
	Then, in these terms,
	$$
	\pi_1: (h,(m,hm), z)\mapsto (m,h)\qquad
	\pi_2:  (h,(m,hm), z)\mapsto (m,hz)
	$$
	$$
	F_{21}:[(h,(m,hm), z),g]\mapsto [(h,(m,hm), z),\Ad_hg],
	$$
	so that the above square reads
	$$
	\xymatrix{
		[(h,(m,hm), z),g]\ar@{|->}[r]\ar@{|->}[dd]&[(h,(m,hm), z),\Ad_hg]\ar@{|->}[d]\\
		&[(h,(m,hm), z),\Ad_hg]\ar@{=}[d]^{(*)}\\
		[(hz,(m,hm), z),g]\ar@{|->}[r]&[(h,(m,hm), z),\Ad_{hz}g],
	}
	$$
	where  $(*)$ follows from commutativity of $(C_{\theta})_m$. 
	
	\end{proof}
	\begin{cor}\label{cor:local_gerbe}
	The Hitchin map $\stack{\mr/\Gtheta}\to\la/W(\la)$ is a neutral gerbe banded by $J_\theta$.
	\end{cor}
	\begin{proof}
		Since the inertia stack descends to $J_\theta\to\la/W(\la)$, it follows that the stack is locally connected. We need to prove that it admits a section, which yields local non emptiness and neutrality at once.  Let  $s_{KR}$ be the Kostant--Rallis section, that is, the section of the Chevalley morphism \eqref{eq Chevalley r} constructed in  \cite[Theorem 11]{KR71} for groups of the adjoint type and adapted in \cite[Theorem 4.6]{HKR} to strongly reductive real groups.  This induces a universal object of our gerbe by assigning to each scheme $f: S\to \laC/W(\la)$ the trivial $G_\theta$ bundle together with the section $\phi: S\times G_\theta\to\mr$ sending $(x,g)\mapsto g\cdot s_{KR}(f(x))$
	\end{proof}
	\begin{rk}\label{rk:HKR_factors_through_real}
		The Kostant--Rallis section factors through the atlas $\mr$ and thus factors through the image of $\stack{\mr/H}\longrightarrow\stack{\mr/\Gtheta}$.
	\end{rk}
	\subsection{Involutions on the local stack}\label{sec:local_invols}
	A way to retrieve local $G_\R$-Higgs bundles from $\stack{\mr/\Gtheta}$ is by studying fixed points by involutions. Indeed, we note that the substack of $G_\R$-Higgs bundles is contained in 
	$\stack{\mr/\Gtheta}^\Theta$, where $\Theta$ is the involution sending
	$$
	\Theta:(E,\phi)\mapsto (E\times_\theta \Gtheta,\phi).
	$$
	
\begin{prop}\label{prop local gerbe}
	The involution on $C_\theta$ given by $\Theta(x,g)=(x, g^\theta)$ descends to an involution $\Theta$ on $J_\theta$ making the isomorphism 
	$$
	\stack{\mr/\Gtheta}\to BJ_\theta
	$$ 
	$\Theta$-equivariant. 	In particular, the image of $\stack{\mr/H}\to\stack{\mr/\Gtheta}$ is contained inside  
	$$
	(BJ_\theta)^\Theta=\{P\in BJ_\theta\,:\, P\times_\Theta J_\theta\cong P\}.
	$$
\end{prop}
\begin{proof}
Descent of the involution follows from $\Gtheta$-equivariance; the induced involution on
 $J_\theta$ (that we will also denote by $\Theta$) is given  by $s^{\Theta}(x)=\theta(s(x))$
for all $s\in J_\theta(S)$. 

The second statement follows by definition of the inertia stack, as sending  $(E,\phi)$ to $(E\times_\theta H,\phi)$ translates into changing the action of an automorphism by $\theta$. Since the image of $\stack{\mr/H}$ is fixed by the involution, it follows that the image will be contained in $(BJ_\theta)^\Theta$.
\end{proof}

	\begin{rk}\label{rk:essential_image}
		Note that the essential image of $\stack{\mr/H}$, i.e., the minimal stack containing all objects of $\stack{\mr/H}$ and those isomorphic to them, is the whole $\stack{\mr/G_\theta}$. This is immediate from commutativity of the local Hitchin map with extension of the structure group,  and the facts that $\stack{\mr/G_\theta}$ is a gerbe over the Hitchin base and that the latter classifies isomorphism classes.
	\end{rk}


Note that the scheme $ J_\theta$ is very twisted and hard to work with in concrete examples. We next give an alternative description of it, also keeping an eye on $BJ^\Theta_\theta$ . Since the right objects are much more clearly understood when looking at the semisimple locus, we will restrict to semisimple Higgs bundles first, then go back to the general case.

\subsection{An alternative description of the band: the semisimple locus}\label{sec:ss_local_band}
In this section we focus on bundles $(E,\phi)$ such that $\phi$ takes only semisimple values. Namely, sections of the stack $\stack{\mrs/\Gtheta}$, where $\mrs\subset\mr$ denotes semisimple elements of $\mr$. This will allow us to find alternative descriptions of $J^\Theta_\theta,\ J_\theta$ over a suitable sublocus of $\la/W(\la)$. 

Let $\lar=\la\cap\mrs$. We have a commutative diagram
\begin{equation}\label{eq triangle ss locus}
	\xymatrix{
		\lar\ar[r]^i\ar[dr]_{\pi_\la}&\mrs\ar[d]^{\pi_\m}\\
		&\lar/W(\la)}.
\end{equation}
Note that
\begin{equation}\label{eq Ttheta}
	i^*C_\theta=\lar\times T_\theta
\end{equation}
where $T_\theta=T\cap \Gtheta$.
\begin{prop}\label{prop local ss gerbe}
	Let $N_\theta(\la)$ be the normaliser of $\la$ inside $\Gtheta$.
	The embedding $\lar\plonge\mrs$ induces an isomorphism 
	$$\stack{\lar/N_\theta(\la)}\cong\stack{\mrs/\Gtheta}$$ of stacks over $\lar/W(\la)$.
	\end{prop}
\begin{proof}
	Let $\stack{\chi}_{\la}:\stack{\lar/N_\theta(\la)}\to\lar/W(\la)$ be the restriction of $\stack{\chi}$ definde in \eqref{eq lshittheta}.  We claim that
	$\stack{\chi}_{\la}$ is a subgerbe of $\stack{\chi}$ with the same band. 
	
	To see it is a gerbe, we check local connectedness and non emptiness.
	
	Local non emptiness of $\stack{\chi}_{\la}$ follows from the fact that given a principal $N_\theta(\la)$-bundle $P\to S$, locally $P\times_{\lar/W(\la)}\lar\cong S\times \lar/W(\la)$.  Indeed, this is a consequence of $W(\la)=N_H(\la)/C_H(\la)$ and Corollary \ref{lm exact seq Ctheta} \eqref{mcFrs}.
	
	Local connectedness follows from local connectedness of $\stack{\mrs/\Gtheta}$ and the fact that two elements $x,y\in \lar$ conjugate by $\Gtheta$ must be conjugate by elements in $N_\theta(\la)$, as by regularity $\lc_\la(x)=\la$ (cf. Definition \ref{rk mr c gr}).
	
	Commutativity of diagram (\ref{eq triangle ss locus}) also implies that both gerbes are locally isomorphic over $\lar/W(\la)$. Indeed, their bands descend from $C_\theta|_{\mrs}$ and $\lar\times T_\theta$ (by Proposition \ref{prop Ctheta descends} and similar arguments for the substack $\stack{\lar/N(\la)}$), which in turn descend to the inertia stacks; since by commutativity of diagram (\ref{eq triangle ss locus}) these descended schemes are isomorphic to $J_\theta|_{\lar/W(\la)}$, it follows that both stacks are  locally isomorphic over $\lar/W(\la)$. 

 But since the band is abelian and any such gerbe is a torsor over a lift of the band, since the local isomorphism of gerbes is globally defined it must be a global isomorphism.
\end{proof}
Through Proposition \ref{prop local ss gerbe} we get a clearer picture of the sheaf $J_\theta$:
\begin{cor}\label{cor Jthetas tori}
	Let $J_\theta$ be the group scheme defined in Proposition \ref{prop Ctheta descends}. 
	Then, its restriction to $\lar/W(\la)$ has $S$ points (for $b:S\to\lar/W(\la)$)		
	$$J_\theta(S)=\Hom_{W(\la)}(\ol{S}_b,T_\theta),$$
	where 
	\begin{equation}\label{eq:cam_cov_real}
\ol{S}_b=S\times_{\la/W(\la)}\la
	\end{equation}
	
	Moreover, $J_\theta^\Theta|_{\lar/W(\la)}$ descends from  $C^\theta|_{\mrs}$ for $C^\theta$ defined in \eqref{defi Ctheta} or, equivalently, from $J^\theta|_{\mathbf{A}\times_{\la/W(\la)}\lar/W(\la)}$ (cf. Lemma \ref{lm gerbe 1}).
\end{cor}
\begin{proof} 
	From Proposition \ref{prop local ss gerbe} we have that
	$$
	J_\theta(U)=\Hom_{N_\theta(\la)}(\ol{U}_b, T_\theta).
	$$
	Since $W(\la)=N_\theta(\la)/T_\theta$ (by \cite[Proposition 7.49]{K}, quasi-splitness of $G_\R$ and the fact that $T_\theta$ is the trivial extension of $T\cap H$ by a central subgroup of $G$), it follows that the action of $N_\theta(\la)$ on $\ol{U}_b$  and $\Ttheta$ factors through the quotient. 
		
	Descent of $C^\theta|_{\mrs}$ follows from Lemma \ref{lm exact seq Ctheta}, which implies that $F^2$ acts trivially on semisimple orbits. Equivalence with descent of $J^\theta|_{\mathbf{A}\times_{\la/W(\la)}\lar/W(\la)}$ is a consequence of Lemma  \ref{lm gerbe 1}. 
\end{proof}
\begin{cor}\label{cor:local_hitchin_fibres}
The fibre of the gerbe $\stack{\mrs/\Gtheta}$  over $b:S\to \lar/W(\la)$ is the category of $T_\theta$ principal bundles $P$ over $\ol{S}_b$ as in \eqref{eq:cam_cov_real} satisfying $w^*P\times_wT_\theta\cong P$ for all $w\in W(\la)$.  Likewise $\stack{\mrs/H}$ is a gerbe over $\lar/W(\la)$,  whose fibre over $b:X\to \lar/W(\la)$ is the category of principal $D$-bundles $P$ over $\ol{X}_b$ satisfying $w^*P\times_wD\cong P$ for all $w\in W(\la)$.
\end{cor}
\subsection{An alternative description of the band: back to arbitrary Higgs fields}\label{sec:local_band}
In order to extend the results of Section \ref{sec:ss_local_band}, we will compare it with the complex case studied  in \cite{DG, NgoLemme}. This will allow us to characterise $J_\theta$ as a group of tori over the whole $\la/W(\la)$ (compare with Corollary \ref{cor Jthetas tori}) and obtain a cocyclic description of the Hitchin fibres.

By quasi-splitness, $\mr\subset\gr$; this embedding is $H$-equivariant, and hence it induces a morphism on the
level of stacks
\begin{equation}\label{eq kappa}
\stack{\mr/H}\stackrel{\kappa}{\longrightarrow}\stack{\gr/G}.
\end{equation}
This morphism factors through $\stack{\mr/\Gtheta}$ (and $\stack{\mr/\Gtheta}^\Theta$), which, 
as Example \ref{ex orbits sl2} illustrates (see also Remark \ref{rk:essential_image})  
a minimal subgerbe containing the image of $\stack{\mr/H}$. Due to this, the set of isomorphism classes of the image
of $\stack{\mr/H}$ inside $\stack{\mr/\Gtheta}$ embeds into the set of isomorphism classes of objects in $\stack{\gr/G}$, unlike what happens for $\stack{\mr/H}$.

\begin{lm}\label{lm compare orbits}
	Let $G_{\R}\leq G$ be a quasi-split real form. Then:
	
 1. There is an equality $N_{G}(\mr)=\Gtheta\subset N_{G}(H)$.
 
 2. We have an embedding 
 \begin{equation}\label{eq embedding GITs}
 	\iota: \la/W(\la)\plonge \lt/W.
 \end{equation}
 
 3. Given $x,y\in\mr$ , if for some $g\in G$ $\Ad_g x=y$, then there is $h\in \Gtheta$ such that
 $\Ad_h x=y$. If $x,y\in\mrs$ , then $h$ can be taken inside of $H$.
\end{lm}
 \begin{proof}
\textit{1.} Note that $g\in N_{G}(\mr)$ if and only if $g^{-1}g^\theta\in C_{G}(x)$ for all $x\in\mr$. But since $\mr$ contains both semisimple and nilpotent elements, 
the intersection of all such centralisers is the center $Z(G)$. So
$N_{G}(\mr)\subset\Gtheta\subset N_G(H)$ and the statement follows.   
 
\textit{2.} By  \cite[Theorem 11]{KR71}, the choice of a principal normal triple $\{e,f,x\}$, with $e,f\in \mr$ nilpotent, establishes an isomorphism $\la/W(\la)\cong f+\lc_\m(e)$. By quasi-splitness, $e,f\in\gr$, and so
$\lt/W\cong f+\lc_\g(e)$ by \cite[Theorem 7]{Kos}, so we have the desired embedding. 

\textit{3.} Follows from \textit{2.} above and \cite[Theorem 7]{KR71}, which implies that $\la/W(\la)$ parametrizes $\Gtheta$-orbits, while $\lt/W$ parametrizes $G$-orbits, by
 \cite[Theorem 2]{Kos}. The statement about semisimple elements follows from \cite[Theorem 1]{KR71}.

 \end{proof}
With respect to the local Hitchin maps, there is a commutative diagram
\begin{equation}\label{eq lshitt}
  \xymatrix{\stack{\mr/H}\ar[r]^\kappa\ar[d]_{\lshit}&\stack{\gr/G}\ar[d]^{\lshitc}\\
  \la/W(\la)\ar@{^(->}[r]_\iota&\lt/W.}
 \end{equation}
 In the above, ${\lt}=\ld\oplus{\la}$ is a maximal
$\theta$-anisotropic Cartan subalgebra of $\g$, that is, a $\theta$-invariant Cartan subalgebra containing
$\la$, and $\ld=\lt^\theta$. The lower horizontal arrow is an embedding by Lemma \ref{lm compare orbits}.
Recall from  Lemma 2.1.1. in \cite{NgoLemme}, that the scheme of centralisers 
$C\subset\gr\times G$ defined analogously to $C^\theta$ (cf. Definition \ref{defi Ctheta})
descends to a scheme of abelian groups $J\to\lt/W$. Using adjunction, one sees that its $S$ points are given by
\begin{equation}\label{eqJ}
J(S)=\Hom_G(S\times_{\lt/W}\gr, C).
\end{equation}
The group scheme $J$ is isomorphic to another group scheme whose $S$ points are given by
\begin{equation}\label{eq mcT}
	\mc{T}(S)=\left\{f: \hat{S}\to T\ \left|
	\begin{array}{l}
	W-\mathrm{equivariant}\\
	\alpha(f(x))=1\textrm{ if }s_{\alpha}(x)=x,\alpha\in\Delta(\g,\lt)
	\end{array}\right.
	\right\}
\end{equation}
for $S\to \lt/W$ and $\hat{S}:= S\times_{\lt/W}\lt$.

On $\la/W(\la)$ we consider the following sheaf of groups: 
\begin{equation}\label{eq mcTtheta}
\mc{T}_\theta(S)=\left\{f: \ol{S}\to T_\theta\ 
\left|
\begin{array}{l}
W(\la)-\mathrm{equivariant}\\
\mathrm{satisfying }\ $\eqref{eq:cond_ram_Ttheta}$,
\end{array}\right.\right\}
\end{equation}
where
\begin{equation}\label{eq:cond_ram_Ttheta}
w(x)=x\then f(x)\in (T_\theta^w)^0\tag{$\dagger$}
\end{equation}
In the above $\ol{S}:=S\times_{\la/W(\la)}\la$ is considered as a subscheme of $\hat{S}:=S\times_{\lt/W}\lt$ and $(T_\theta)^0$ is the identity component of $\Ttheta$.
\begin{prop}\label{prop:Ttheta_representable}
	The sheaf $\mc{T}_\theta$ is the intersection of the Weil restriction of the torus $\laC\times T_\theta$ along the finite flat morphism $\laC\longrightarrow\laC/W(\la)$  and $\iota^*\mc{T}$, where $\iota$ is defined in \eqref{eq embedding GITs}.  In particular, it is representable by a scheme of tori over $\laC/W(\la)$.
\end{prop}
\begin{proof}
	Representability of the Weil restriction  follows by  \cite[\S 10]{BLR}, and that of $\iota^*\mc{T}$ by \cite[Proposition 2.4.7]{NgoLemme}, hence, the intersection is also representable. So representability follows by proving that  $\mc{T}_\theta$ is the intersection of both.
	
	Now, by adjunction
	$
	\iota^*\mc{T}(U)
	$
	consists of $W$-equivariant morphisms $U\times_{\la/W(\la)}\lt\to T$ satisfying $\eqref{eq:cond_ram_Ttheta}$ for reflections along roots. By Theorem \ref{thm DG cameral data}, the same must be true for any element of the Weyl group, by the way the ramification divisors are defined. Thus, the intersection with the Weil restriction of $\la\times T_\theta$ consists of morphisms $U\times_{\la/W(\la)}\lt\to T$  satisfying $\eqref{eq:cond_ram_Ttheta}$ arising from $W(\la)$-equivariant morphisms. Hence, all we need to prove is that extension of sections of $\mc{T}_\theta$ by $W$-equivariance is well defined and injective. 
	
	Given $w\in W$, let $w\cdot x\in \laC$ for some $x\in \laC$. Then, there exists some $w'\in W(\la)$ such that $w'x=wx$, namely
	$$
	w^{-1}w'\cdot x=x.
	$$
	The condition $\eqref{eq:cond_ram_Ttheta}$ ensures that $s(x)\in (T_\theta)^{w^{-1}w'}$, thus we may unambiguoulsy define 
	$$
	s(w\cdot x)=w\cdot s(x)
	$$
	for any $w\in W$ and $x\in\la$. This forces to any two distinct sections to extend differently, so injectivity follows.
\end{proof}
\begin{prop}\label{prop:Jtheta_embeds_in_J}
	We have an embedding $J_\theta\subset \iota^*J$, where $\iota$ is defined in \eqref{eq embedding GITs}.
\end{prop}
\begin{proof}
Consider the commutative diagram
$$
\xymatrix{
	\mr\ar[r]^{\tilde{\iota}}\ar[d]_\chi&\gr\ar[d]^\chi\\
		\la/W(\la)\ar[r]_{\iota}&\lt/W.
	}
$$	
Then $C|_{\mr}:=\tilde{\iota}^*C=\tilde{\iota}^*\chi^*J=\chi^*\iota^*J$. So given that $C_\theta\subset C|_{\mr}$, adjunction gives a morphism
	$J_\theta\to \chi_*\tilde \iota^*C$. Using the fact that adjunction yields isomorphisms between triple combinations, we have that on the level of $S$-points the morphism sends a section of $J_\theta(S)$ to a section of $\iota^*J(S)$ by establishing $s(g\cdot x)=g\cdot s(x)$. This is well defined, as if
	$g\cdot x\in\mr$ for $x\in\mr$ then there exists some $g'\in G_\theta$ such that
	$g'\cdot x=y$, so that $s(y)=g'\cdot s(x)$, and by commutativity of the centralisers, $s(y)=g\cdot s(x)$, so  extensions are well defined. 
	 
	 
	 Injectivity follows from left exactness of pullback, or by tracking all morphisms involved, which are injective.
\end{proof}
On ${\g}\times {G}$ we define the involution $\Theta$ by
\begin{equation}\label{eq Theta}
\Theta:(x,g)\mapsto (-\theta x,g^\theta). 
\end{equation}
The subscheme of fixed points is $\m\times H$. The restriction  of $\Theta$ to $\lt\times T$ also induces an involution whose fixed point set is $\la\times D$.
 \begin{prop}\label{thm compare sheaves} 
 	\begin{enumerate}	
 		\item\label{descent theta J} The involution $\Theta$ descends to an involution on $J\to \lt/W$  (denoted also by $\Theta$)
 		such that $\iota^*J^\Theta\cong J^\Theta_\theta$.
 		
 			\item\label{descent theta T} The restriction of $\Theta$ to $\lt\times T$ descends to an involution  (also denoted by $\Theta$) on $\mc{T}\to \lt/W$ such that $\mc{T}_\theta^\Theta\cong\mc{T}^\Theta$.
 		
 		\item\label{theta equivar J=T} The isomorphism $J\cong \mc{T}$ is $\Theta$ equivariant. In particular
 		$$
 		{J}^\Theta\cong\mc{T}^\Theta.
 		$$ 
 		\end{enumerate}. 
	\end{prop}
\begin{proof}
\eqref{descent theta J}  Note that the source and target maps $s,\ t:\gr\times G\to\gr$ are $(\Theta,-\theta)$ equivariant. So $\Theta$ induces an action on $J$ (see equation (\ref{eqJ})) such that $J\to\lt/W$ is $(\Theta,-\theta)$-equivariant (the action on $\lt/W$ follows from stability of orbits by the action of $\theta$). By Proposition \ref{prop:Jtheta_embeds_in_J} and compatibility of the involutions, we have $J_\theta^\Theta\subset \iota^*J^\Theta$. For the inverse,  by seeing  
	$\iota^*J(S)\subset \Hom_{\Gtheta}(S\times_{\la/W(\la)}\g, C)$, we have that pullback by the embedding  $i_S:S\times_{\la/W(\la)}\mr\plonge S\times_{\lt/W}\gr$ (cf. Lemma \ref{lm compare orbits}) induces a morphism 
	$$
	i_S^*: i^*J(S)^\Theta\to J_\theta^\Theta(S)
	$$
	 which follows by definition of the involutions taking into account that  $\la/W(\la)=(\lt/W)^{-\theta}$. 
	 
	 Both maps are clearly inverse, and so the conclusion follows.

\eqref{descent theta T}  The same arguments as in \eqref{descent theta J} (see also Remark \ref{rk:embeding_Ttheta_T}) imply that $\Theta|_{\lt\times T}$ induces an involution $\Theta$ on $\mc{T}$. Likewise, for each  $S\to\la/W(\la)$, there is an embedding $i_S:S\times_{\la/W(\la)}\la\plonge S\times_{\lt/W}\lt$ by Lemma \ref{lm compare orbits}, and restriction by $i_S$ induces a morphism $i_S^*: \iota^*\mc{T}^\Theta(S)\to \mc{T}^\Theta_\theta(S)$. The conditions imposed on ramification points ensure that any section $f\in  \mc{T}(S)^\Theta_\theta(S)$ can be uniquely extended to a
    section of $\iota^*\mc{T}^\Theta(S)$ by setting $f(wx)=wf(x)$ for each $w\in W$. Injectivity follows because $\mc{T}^\Theta(S)$ is contained in the subset 
    $$
  \{ s\in \mc{T}(S)\,:\, w(s(x))=s(x) \,\pt w\in C_W(x), x\in \laC\}.
    $$

\eqref{theta equivar J=T} As for checking that the isomorphism $J\to \mc{T}$ respects the involutions, it follows from $\theta$-equivariance of the morphisms $C\to \ol{B}$ and $\ol{B}\to \hat{\g}_{reg}\times T$ where  $\ol{B}\to \hat{\g}_{reg}$ is the sheaf of Borel subgroups determined by the choice of a $\theta$ anisotropic  Borel subgroup $B$ containing $T$. Note that although $B$ is not $\theta$ invariant, $B/[B,B]$ is and $B/[B,B]\cong T$ is $\theta$-equivariant. See \cite[Proposition 2.4.2]{NgoLemme} for details. The second statement is immediate from this.	
\end{proof}
The following theorem is the key towards a cocyclic description of $BJ_\theta$ (and thus, of the fibres of the local Hitchin map \eqref{eq lshit}):
\begin{thm}\label{thm Jtheta=Ttheta}
	We have $\mc{T}_\theta\cong J_\theta$.
\end{thm}
\begin{proof}
	From Proposition \ref{thm compare sheaves} \eqref{descent theta J} and \eqref{theta equivar J=T}, it is enough to prove that the isomorphism 
	$J\cong\mc{T}$ takes $J_\theta$ to $\mc{T}_\theta$. The isomorphism is $\Theta$ equivariant by Proposition \ref{thm compare sheaves} \eqref{theta equivar J=T}, which implies that that $C_\theta|_{\mr}$ descends to the same group scheme as $\chi^*\iota^*\mc{T}\cap G_\theta$. By tracing back the construction of $\mc{T}$ (see the proof of \cite[Proposition 2.4.2]{NgoLemme}), we see that this is equivalent to having that the restriction of $C_\theta$ to $\hat{\m}_{reg}:=\mr\times_{\la/W(\la)}\lt$ being mapped to the intersection of $\hat{\m}_{reg}\times \Ttheta$ and $\mc{T}|_{\hat{\m}_{reg}}$. But by Proposition \ref{prop:Ttheta_representable},  the latter group scheme and  $\mc{T}_\theta$ match, and so the isomorphism follows.
\end{proof}
\begin{rk}\label{rk regular qsplit}
	If $G_{\R}$ is such that $C_W(x)\subset W(\la)$ for all $x\in\la$, then we can alternatively define $\mc{T}_\theta$ as
	$$
	\mc{T}_\theta(S)=\left\{f: S\times_{\la/W(\la)}\la\to T_\theta\ \left|\begin{array}{l}
	W(\la)-\mathrm{equivariant}\\
	\alpha(f(x))=1\\
	\textrm{ if }s_{\lambda}(x)=x,\lambda=\alpha|_{\la}\in\Sigma(\la).
	\end{array}\right.\right\}
	$$
	There are two main differences with respect to \eqref{eq Ttheta}: on the one hand, condition $\eqref{eq:cond_ram_Ttheta}$ is substituted by a condition involving only restricted roots. The second difference, a consequence of the former, is that the cover $\ol{S}$ need not be seen as a subscheme of $\hat{S}$. It therefore becomes a statement about restricted cameral covers only. 
	Simple groups satisfying the former condition are all the split forms and $\SU(p,p)$.  They can be characterised by the fact that the set $\la_{sing}$ of points on which more than one restricted root $\lambda\in\Sigma(\la)$ vanishes is the intersection $\lt_{sing}\cap\la$ of points of $\la$ on which more than one root vanishes.  In terms of spectral curves, these are the ones with generically smooth spectral curves.
\end{rk}

\begin{rk}\label{rk:cam_cov+inv}
	The content of Proposition \ref{thm compare sheaves} can be reformulated in terms of descent to the stack $\stack{\la/W(\la)/\Z_2}$, where the action of $\Z_2$ on $\la/W(\la)$ is trivial. Indeed, on $\lt$ we can consider the action of the group $W_\Theta=W\rtimes\Theta$, where the action of $\Theta$ on $W$ is given by
	$w^\Theta(x)=\theta(w(-\theta(x)))$. Then $j:\stack{\la/W(\la)/\Z_2}\plonge\stack{\lt/W/\Z_2}$. The schemes
	$J_\theta$ and $J$ descend further to  $\ol{J}_\theta\to\stack{\la/W(\la)/\Z_2}$ and $\ol{J}\to\lt/W_\Theta$ respectively, and Proposition \ref{thm compare sheaves} says precisely that $j^*\ol{J}_{(\theta)}\cong \ol{J}_{(\theta)}$. Note that $\stack{\lt/W/\Z_2}$ parametrizes pairs of cameral covers related by an order two morphism, while  $\stack{\la/W(\la)/\Z_2}$ can be identified with the set of cameral covers
	together with involutions.
\end{rk}

\section{The global situation: twisted Higgs bundles}\label{sec:global}
In this section we study the structure of the Hitchin map $\shit_L$ defined in (\ref{eq hitchin real}) 
from that of $\lshit$ from Equation \eqref{eq lshit} analysed in Section \ref{sec:local_Higgs}.

Let us recall the notation: we fix $X$ a smooth complex projective curve of genus $g\geq 2$, and $L\to X$  an \'etale line bundle, 
that we will assume to be of degree at least $2g-1$ or equal to the canonical bundle $K$.
Let $\Higgsr(G_{\R}):=\stack{\mr\otimes L/H}$ and $\Higgsrt(G_{\R}):=\stack{\mr\otimes L/\Gtheta}$. Recall the Hitchin map $\shit_L$ defined in \eqref{eq hitchin real} and denote by $\shit_{L,\theta}$ its natural extension to
$\Higgsrt(G_{\R})$ (also induced by $\lshit$ as in \eqref{eq lshit} by noticing it is $\Gtheta$ equivariant). Let $J^\theta\to \mathbf{A}$ and $J_\theta\to\la/W(\la)$ be the regular centraliser schemes defined through Lemma \ref{lm gerbe 1} and Proposition \ref{prop Ctheta descends}. Similarly, we have sheaves of tori $\mc{T}_\theta\to \la/W(\la)$ and $\mc{T}\to\lt/W$ defined in \eqref{eq mcTtheta} and \eqref{eq mcT}
\begin{lm}\label{global band from local}
	The group scheme $J^\theta$ (resp. $J_\theta$) descends to group schemes $J^\theta_\C$ (resp. $J_{\theta,\C}$) over 
	$\stack{\mathbf{A}/\C^\times}$ (resp. $\stack{\la/W(\la)/\C^\times}$) whose respective pullbacks by 
	$[L]:X\to B\C^\times$ we denote by $J^\theta_L$ and $J_{\theta,L}$. Similar statements hold for $\mc{T}$ and $\mc{T}_\theta$, which yield group schemes $\mc{T}_{L},\, \mc{T}_{\theta,L}$ over $\mc{A}_L(G_\R)$. 
\end{lm}
\begin{proof}
	Descent follows from $\C^\times$-equivariance of $C^\theta\to\mr$,  $C_\theta\to\mr$, $\la\times T\to\la$ and $\la\times T_\theta\to\la$. 
\end{proof}
A consequence of the above is that the Hitchin map induces a gerbe structure on the corresponding stack:
\begin{thm}\label{gerbe Higgs bundles}
	Let $({G_{\R}},H,\theta,B)<({G},U,\tau,B)$ be a quasi-split real form of a connected complex reductive
	algebraic group. Let $X$ be a smooth complex projective curve, and let  
	$L\to X$ be an \'etale line bundle. Then, 
	\benum
	\item\label{thm Gtheta is gerbe} The Hitchin map
	\begin{equation}\label{eq shittheta}
	\shit_{L,\theta}:\Higgsrt(G_{\R})\to\la\otimes L/W(\la)  
	\end{equation}
	is a gerbe banded by the abelian group scheme $J_{\theta,L}$. 
	\item\label{thm Gtheta is neutral gerbe}  Moreover, $\Higgsrt(G_{\R}) \cong BJ_{\theta,L}$ when the degree of $L$ is even. 
	\item\label{thm G is neutral gerbe}  Likewise $\shit_{L}:\Higgsr(G_{\R})\to \mathbf{A}_L:= \mathbf{A}\times_{[L]}B\C^\times$ is a $J_L^\theta$-banded gerbe which is neutral whenever $L$ has even degree. 
	\eenum
\end{thm}
\begin{proof}
	For \eqref{thm Gtheta is gerbe}, note that $\stack{\mr/{G_\theta}\times \C^\times}$ is a gerbe banded by $J_{\theta, \C}$.
	Indeed, it is locally non empty over
	$\stack{{\la}/{G_\theta}\times \C^\times}$, since any covering of ${\la}/W({\la})$ 
	over which $\stack{\mr/{H}}$ is non-empty
	is a cover for $\stack{{\la}/W({\la})/{G_\theta}\times \C^\times}$ over which $\stack{\mr/{G_\theta}\times \C^\times}$ is non-empty. Local connectedness follows in the same way. Clearly, inertia descends
	to $J_{\theta,\C}$, by $\C^\times$ equivariance of $C^\theta\to\mr$, which implies that the band is $J_{\theta,\C}$. 
	
	On the other hand, the diagram
	$$
	\xymatrix{
		\Higgsrt(G)\ar[rr]\ar[dr]\ar[ddr]&&\stack{\m/{\Gtheta}\times \C^\times}\ar[dr]\ar[ddr]&\\
		&\la\otimes L/W(\la)\ar[d]\ar[rr]&&\stack{\la/W(\la)/\C^\times}\ar[d]\\
		&X\ar[rr]_{[L]}&&B\C^\times
	} 
	$$
	is Cartesian, so for (\ref{eq shittheta}) to define a gerbe it is enough to check non-emptiness, as
	the remaining structure is preserved by pullback. Local
	non emptiness is a consequence of local triviality of $L$ and Corollary \ref{cor:local_gerbe}. The statement
	about the band follows from Lemma \ref{global band from local}.
	
	Statement \eqref{thm Gtheta is neutral gerbe} follows from 
	Theorem 6.13 in \cite{HKR}, where the existence of a section of the Hitchin map is proved. 
	Note that the construction in \cite{HKR} goes through to the stacky setting.
	
	Finally, \eqref{thm G is neutral gerbe} follows by the same arguments as \eqref{thm Gtheta is gerbe}. Neutrality is a consequence of the fact that the Hitchin--Kostant--Rallis section from \cite{HKR} factors through the common atlas $(\m\otimes L)_{reg}$.
\end{proof}
\subsection{Cameral data}

We next proceed to the description of the fibres of the Hitchin map \eqref{eq shittheta} in terms of cameral data.

\begin{defi}
We define the cameral cover associated with $b\in H^0(X,\laC\otimes L/W(\la))$ to be the ramified $W$-Galois cover  of $X$ fitting in the Cartesian diagram
\begin{equation}\label{eq def cameral cover}
\xymatrix{
	\hat{X}_b\ar[rr]\ar[d]_{p_b}& &\lt\otimes L\ar[d]_{\pi}\\
	X\ar[rr]_{\iota\circ b}&&\lt\otimes L/W.
	}
\end{equation} 
We denote by $\ol{X}_b:=X\times_{b}\la\otimes L$, that is, the subscheme of $\hat{X}_b$ fitting in the Cartesian diagram
\begin{equation}\label{eq def cameral cover real}
\xymatrix{
	\ol{X}_b\ar[rr]\ar[d]_{q_b}& &\la\otimes L\ar[d]\\
	X\ar[rr]_{b}&&\la\otimes L/W(\la)
}
\end{equation} 
\end{defi}
\begin{rk}\label{rk:cam_subcover}
Note that  $\ol{X}_b$ is only a $W(\la)$-subcover if ramification is determined by $W(\la)$, namely, if $C_W(x)\subset W(\la)$ for all $x\in\laC$. By Remark \ref{rk regular qsplit}, in the simple group case this happens for split groups and $\SU(p,p)$. 
\end{rk}
\begin{thm}\label{thm fibres hm for Gtheta}
	The fibre of the Hitchin map $\shit_L$ over $b\in H^0(X,\laC\otimes L/W(\la))$ is given by the subcategory $\mc{C}am_\theta$ of $\mc{C}am$ consisting of weakly $R$-twisted, $N$-shifted principal $T$-bundles over $\hat{X}_b$ admitting a reduction of the structure group to $\Ttheta$ over $\ol{X}_b$.
\end{thm} 
\begin{proof}
	Since $\Higgsrt(G_{\R})$ is a subgerbe of $\Higgsr(G)$,  by Theorem \ref{thm DG cameral data} it is enough to identify the points of the stack of $\mc{C}am$ corresponding to $N(G_\R)$-Higgs bundles. In order to do this, by Theorem \ref{gerbe Higgs bundles} \eqref{thm Gtheta is gerbe} and Theorem \ref{thm Jtheta=Ttheta}, it is enough to prove that the subcategory $\mc{C}am_\theta$ is a $\mc{T}_{\theta,L}$ banded gerbe that intersects the image of $\Higgsrt(G_{\R})$ inside of $\Higgsr(G)$. 
	
	We will check non emptiness of $\mc{C}am_\theta$ and that the intersection with $\Higgsrt(G_{\R})$ is non empty at once, by directly checking that a cameral datum of an element in $\Higgsrt(G_{\R})$ satisfies the required condition.
	
	In order to do this, we need a different approach to the cameral cover $\ol{X}_b$, which is given by Proposition \ref{prop:alternative_cam_cover}. According to this, a cameral cover is, \'etale locally, the cover 
	$$
	\HN\to\HC,
	$$
	where $\HC$ parametrises regular centralisers and $\HN\subset \HC\times\Gtheta/\Ttheta$ is the incidence variety. We recall that by Proposition \ref{prop:var_min_par}  $\Gtheta/\Ttheta$ parametrises $\theta$-anisotropic Borel subgroups.
	
	Recall now \cite{DG} that locally, the cameral datum is constructed by pullback to $\GT$ of the torus bundle
	$$
	G/U\to G/B.
	$$
Over the subvariety $\Gtheta/\Ttheta$, we have the principal $\Ttheta$-bundle
$$
\Gtheta\to\Gtheta/\Ttheta.
$$
This fits in a commutative diagram
$$
\xymatrix{
	\Gtheta\ar@{^(->}[r]\ar[d]&G/U\ar[d]\\
\Gtheta/\Ttheta\ar@{^(->}[r]_{j}&G/B.	}
$$
Thus $\Gtheta(T)\cong j^*G/U$. Since the universal cameral datum is pullbacked from $G/U$, Proposition \ref{prop:alternative_cam_cover} allows us to conclude that $\Higgsrt(G_\R)\cap\mc{C}am_\theta\neq \emptyset$.
	
	Now, by local connectedness of $\mc{C}am$ it is enough to compute the automorphism sheaf of elements in $\mc{C}am_{\theta}$. These are clearly sections of $\mc{T}_L$ which over the smaller cameral cover reduce to automorphisms of the $T_\theta$-bundle, namely, sections of $\mc{T}_{\theta,L}$.
\end{proof}

 \subsection{Cameral data for fixed points under involutions}
 We have already observed (cf. Proposition \ref{prop local gerbe}) that a natural way to obtain candidates for $G_{\R}$-Higgs bundles in terms of cameral data is to understand the involution on $\mc{C}am$  induced from the involution on $\Higgsr(G)$ sending 
 \begin{equation}\label{eq:inv_Higgs}
 \Theta:(E,\phi)\mapsto (E\times_\theta G,-\theta(\phi)).
 \end{equation}

 We start by some Lie theoretic preliminaries.  Let $B<G$ be a $\theta$-anisotropic Borel subgroup (cf. Definition \ref{def:min_aniso}). Note that $B$ is associated to an ordering of the roots such that $\la^*>i\ld^*$.  
 
 The action of $-\theta$ on $\g$is induces an action:
\begin{align}
\label{eq:Theta_GB}
-\theta: G/B\longrightarrow G/B,&& \lb\mapsto\lb^{-\theta}\\
	-\theta: \GN\longrightarrow\GN,&& \lc\mapsto\lc^{-\theta}\\\label{eq:Theta_GT}
	-\theta: \GT\longrightarrow \GT,&& (\lc,\lb)\mapsto(\lc^{-\theta},\lb^{-\theta}),
\end{align}
which makes the universal cameral cover $\GT\to\GN$ $-\theta$ equivariant.
\begin{lm}\label{lm:fixed_pts_GN}
	We have $G/B^{-\theta}\cong \Gtheta/\Ttheta\cong H/\CH$. In particular:
$$\mr\times_{\HN}\GT^{-\theta}\cong \mr\times_{\HN}\HC.
$$
\end{lm}
\begin{proof}
Note that the second statement is automatic from the first one, $\theta$-equivariance of the morphism $\gr\to\GN$ and Proposition \ref{prop:var_min_par}.

Now, the involution $\theta$ acts on the fixed set of roots by inverting the order. By expressing
	$$
	\lb=\lt\oplus\bigoplus_{\lambda\in \Lambda(\la)^+}\g_\lambda,
	$$
	we may identify the action of $-\theta$ on $G/B$ by
	\begin{equation}\label{eq:ThetaGB}
	-\theta:gB\mapsto g^\theta B.
	\end{equation}
Indeed, this follows from the fact that $\theta$ exchanges $\lt$ root spaces by $\g_\alpha\mapsto\g_{\theta\alpha}$. Thus fixed points of $G/B$ consist of Borel subalgebras corresponding to cosets $gB$ such that $g^{\theta}B=gB$, namely, such that $g^{-1}g^{\theta}\in B$. We need to check that this is equivalent to saying that $(gB)^{\theta}=(gB)^{op}$, i.e., that
$$
g^\theta B^\theta\cap gB=gT.
$$
Now, $B^\theta=B^{op}$, thus $g^\theta B^{op}\cap gB=gT\iff g^{-1}g^\theta B^{op}\cap B=T\iff g^{-1}g^\theta\in B$.
\end{proof}
\begin{rk}\label{rk:borel_and_op}
	From the above proof we deduce that $g^{-1}g^\theta B=B\iff g^{-1}g^\theta\in T$.
\end{rk}
\begin{lm}\label{lm:invol_on_covers}
	Let $(E,\phi)\in\Higgsr(G)$ have associated cameral cover $\hat{X}_b$ for some $b\in H^0(X,\lt\otimes L/W)$. Then, the associated cameral cover to $\Theta(E,\phi)$ is 
	$$
	\hat{X}_{-\theta(b)}= X\times_b({-\theta})^*\lt\otimes L,
	$$
	where $-\theta:\lt\otimes L/W\to \lt\otimes L/W$ is defined by $-\theta\pi(x):=\pi(-\theta x)$ for $x\in\lt\otimes L$, $\pi$ as in \eqref{eq def cameral cover}. In particular, if $b\in H^0(X,\laC\otimes L/W(\la))$ then we have an involutive isomorphism of $W$-covers
		$$
		\hat{X}_{b}\cong  X\times_b({-\theta})^*\lt\otimes L.
		$$
		Equivalently, if $(E,\sigma)$ is the associated abstract Higgs bundle associated with $(E,\phi)$,  with $\sigma:E\to\GN$ given by $\sigma(x,e)=\lc(\phi(x))$, then
		$$
		\hat{X}_{\sigma^{-\theta}}:= X\times_{\sigma^{-\theta}}\GT\cong {X}\times_{\sigma}(-\theta)^*\GT.
		$$
\end{lm}
\begin{proof}
By Fact \ref{fact:springer_res}, cameral covers can be defined in either of these ways. By Lemma \ref{lm:fixed_pts_GN}, the choice of the involutions is compatible, and so both statements are equivalent, so it is enough to prove one of them, for instance, the first, which follows by definition.
\end{proof}
In a similar way, on the stack of cameral data $\mc{C}am$ over $\mc{A}_L(G)$ (see  Definition \ref{def:mcCD}), we consider the following involution: given an $R$-twisted, $N$-shifted $W$-equivariant principal $T$ -bundles $(P,\gamma,\beta)$ over $\widehat{S}_b$, we assign to it
\begin{equation}\label{eq:inv_CD}
\Theta:(P,\gamma,\beta)\mapsto ((-\tilde{\theta})_b^*P,(-\tilde{\theta})_b^*\gamma,(-\tilde{\theta})_b^*\beta).
\end{equation}
In the above, $(-\tilde{\theta})_b$ is locally given by universal version in the Cartesian diagram
\begin{equation}\label{eq:cartesian_camcov_pulbacked}
\xymatrix{
	(-\theta)^*\GT\ar[r]^{-\tilde{\theta}}\ar[d]&\GT\ar[d]\\
	\GN\ar[r]_{-\theta}&\GN	
}.
\end{equation}
\begin{thm}\label{thm:cam_data_fixed_pts}
 	The equivalence $\Higgsr(G)\cong \mc{C}am$ is $\Theta$-equivariant. In particular, the involution \eqref{eq:inv_CD} is well defined and induces an  equivalence
 	 $\Higgsr(G)^\Theta\cong \mc{C}am^{\Theta}$. We may describe
\begin{equation}
\label{eq:fixed_pts}
\mc{C}am^\Theta\cong\{(P,\gamma,\beta)\,:\, P\times_\theta T\cong \eta^*P\},
\end{equation} 	 
where $\eta:\hat{X}_b\to\hat{X_b}$ is the involution naturally induced from the isomorphism $(E,\phi)\cong\Theta(E,\phi)$.
 \end{thm}
 
 \begin{proof}
 	In order to check $\Theta$-equivariance, note first that the involution is well defined. Indeed, it is enough to check this on the universal cameral datum \eqref{eq:universal_cameral_datum}. The fact that the principal $T$ bundle transforms in the stipulated way follows by definition. To see the way $\gamma$ and $\beta$ transform, we note that $-\tilde\theta$ is induced from $(-\theta,id)\laction \GN\times G/B$, and thus it is $W$-equivariant and exchanges the action of $g$ and $g^\theta$ on the first factor. Hence
 	$(-\tilde{\theta})^*\gamma$ induces an isomorphism
 	$$
 	\Theta(P)\cong (-\tilde\theta)^*\left(w^*P\times_{w}T\right)\otimes (-\tilde\theta)^*R_{w}=w^*(-\tilde{\theta})^*P\times_{w}T\otimes R_{w}.
 	$$
 	Finally, let $\beta:=(\beta_i)$ with $\beta_i(n_i):\alpha_i(P)|_{D_{\alpha_i}}\cong\mc{O}(D_{\alpha_i)})|_{D_{\alpha_i}}$ for all $n_i\in N$ a lift of the simple root $\alpha_i$. Then, $W$-equivariance (induced from $G$-equivariance on the $G/B$ factor of $\GT$) allows us to conclude.

 	 Now, since cameral data are fully determined by the universal version once an abstract cameral cover has been assigned,  
 	  equivariance is clear from Lemma \ref{lm:invol_on_covers}.

Let us now prove \eqref{eq:fixed_pts}. Note that the cameral cover associated to a fixed point can be obtained in two ways as illustrated by the following diagram:
\begin{equation}\label{eq:cartesian_cam_covs}
\xymatrix{
	\hat{X}_b\ar[rrrr]^{\tilde\sigma^\theta}\ar[dd]\ar[dr]_\eta&&&&\GT\ar[dd]\ar[rd]^{(-\theta,-\theta)}&
	\\
	&\hat{X}_b\ar[dd]\ar[rrrr]^{\tilde{\sigma}}&&&&\GT\ar[dd]\
	\\
	X\ar@{=}[dr]\ar[rr]^{-\theta\phi}&&E_\theta(\g)\otimes L\ar[dr]_{-\theta}\ar[rr]^{\lc}&&\GN\ar[dr]_{-\theta}&
	\\
	&X\ar[rr]_\phi&&E(\g)\otimes L\ar[rr]_{\lc}&&\GN.
	}
\end{equation}
The fact that the rightmost top arrow is given by $(-\theta,-\theta)$ follows from the fact that it is an isomorphism of universal cameral covers making the diagram commute.



Now, we note that $\Puniv=(-\theta,-\theta)^*\Puniv\times_\theta T$. Indeed, this is a consequence of $(G\times W,G^\theta\times W^\theta)$ equivariance of $(-\theta,-\theta)\laction \GN\times G/B$. Thus,  assume  $\Theta(E,\phi)\cong (E,\phi)$. Then, on the one hand, the cameral datum for $\Theta(E,\phi)$ is $\eta^*\tilde{\sigma}^{*}\Puniv$ (because it is a fixed point).  On the other hand, it is simply $\tilde\sigma^{\theta, *}\Puniv=\tilde{\sigma}(-\theta,-\theta)^*\Puniv\times_\theta T$ which by commutativity of the top square in  \eqref{eq:cartesian_cam_covs} equals
$\eta^*\tilde\sigma^{\theta,*}\Puniv\times_\theta T$. Hence the statement follows. 

\end{proof}
As a corollary we obtain some simple consequences:
\begin{cor}\label{cor:split}
	Let $\theta$ be the involution produced from the split real form. Then $\mc{C}am^\Theta$ are the order two points of $\mc{C}am$.
	\end{cor}
	\begin{proof}
		Note that for split real forms one has $\la=\lt$, so that $\eta=id$ (by Proposition \ref{prop:alternative_cam_cover}) and $\theta$ is ${-1}$ on the split torus $T$. Hence, fixed points are principal bundles on $\hat{X}_b$ (for all $b$) such that
		$
		P^{-1}=P\times_\theta T=P.
		$	\end{proof}
\begin{prop}\label{prop:stable_fixed_points}
	Let $G$ be a simple Lie group. Let $$(E,\phi)\in\Higgsr(G)^\Theta$$ be a stable Higgs bundle. Then, $(E,\phi)\in\Higgsr(G_{\R})$. 
\end{prop}
\begin{proof}
	
	Let $(E,\phi)\in\Higgsr(G)^\Theta$ be a stable bundle. By stability , it is simple and by \cite[Proposition 7.5]{GRInvol} the structure group reduces to a real form with involution in the $\mathrm{Int}(G)$ orbit of $\theta$. By regularity, the form must be quasi-split. But there are at most two non-isomorphic quasi-split real forms, corresponding to an outer involution (the split real form, which always exists) and an inner involution (which may or may not exist). Since both cases are exclusive, the result follows.
	
\end{proof}

\appendix
\section{Lie theory}\label{sec:Lie_th}
In this section we summarize the main results about real reductive groups. The references are \cite{K, KR71, HKR}.

A reductive real Lie group $G_\R$ is a Lie group in the sense of \cite[\S VII.2, p.446]{K},  that is, a tuple
$(G_\R,H_\R,\theta,\langle\,\cdot\,,\,\cdot\,\rangle)$, where $H_\R \subset G_\R$ is a maximal compact subgroup, 
$\theta\colon \g_\R
\to \g_\R$ is a Cartan involution and $\langle\,\cdot\,,\,\cdot\,\rangle$ is a
non-degenerate bilinear form on $\g_\R$, which is $\Ad(G_\R)$-
and $\theta$-invariant, satisfying natural compatibility conditions. 
\begin{defi}\label{def:strongly_red}
	A \textbf{real reductive group} is a $4$-tuple $(G_\R,H_\R,\theta,\langle\,\cdot\,,\,\cdot\,\rangle)$ where
	\benum
	\item $G_\R$ is a real Lie group with reductive Lie algebra $\g_\R$.
	\item $H_\R<G_\R$ is a maximal compact subgroup.
	\item $\theta$ is a Lie algebra involution of $\g_\R$ inducing an eigenspace decomposition
	$$\g_\R=\h_\R\oplus\m_\R$$
	where $\h_\R=\mathrm{Lie}(H_\R)$ is the $(+1)$-eigenspace for the action of $\theta$, and $\m_\R$ is the 
	$(-1)$-eigenspace.
	\item\label{bilinear form} $\langle\,\cdot\,,\,\cdot\,\rangle$ is a $\theta$- and $\Ad(G_\R)$-invariant non-degenerate bilinear form,
	with
	respect to which $\h_\R\perp_{\langle\,\cdot\,,\,\cdot\,\rangle}\m_\R$
	and $\langle\,\cdot\,,\,\cdot\,\rangle$ is negative definite  on $\h_\R$ and positive definite on $\m_\R$.
	\item\label{diffeo} The multiplication map $H_\R\times \exp(\m_\R)\to G_\R$ is a diffeomorphism.
		\item\label{SR}  $G_\R$ acts by  inner automorphisms on the complexification $\g$ of its Lie algebra via the adjoint representation 
	\eenum
		\end{defi}
		 The group $H_\R$ acts linearly on $\m_\R$ through the adjoint
representation of $G_\R$ --- this is the isotropy representation that we
complexify to obtain a representation (also referred as isotropy representation)
\begin{math}
	\rho_i\colon H \to \GL(\m).    
\end{math} 
Let $\g:=\g_{\R}^\C$, and similarly for $\m$ and $\h$. Let $\mr$ denote the set of regular elements of $\m$, namely, elements with maximal isotropy orbits, and by $\mrs$ the set of regular semisimple elements.

 Let  $\lt=\ld\oplus\la$ be a $\theta$-invariant Cartan subalgebra of $\g$ (where $\theta$ denotes the extension by complex linarization of the Cartan involution of $\g_\R$), with  $\la=\lt\cap\m$ maximal and $\ld=\lt\cap\h$.
\begin{prop}{\cite[Theorem 1]{KR71}}\label{prop ss conjugate to a}
 Let $x\in\m$ be semisimple. Then $x$ is $H$-conjugate to an element of $\la$.
\end{prop}
The above proposition fails to be true for non semisimple elements of $\m$. 

\begin{prop}\label{prop orbits in laW}\cite[Theorem 9]{KR71}
	The space $\la/W(\la)$ classifies $(\Ad(G))^\theta$ orbits in $\mr$. 
\end{prop}
\begin{prop}\label{prop KR}
	There exists a section of $\mr\to \la/W(\la)$ intersecting each $G_\theta$-orbit at exactly one point.
\end{prop}
\begin{proof}
	This is Theorem 11 in \cite{KR71} (adjoint group case) and a consequence of the former
	and Theorem 4.6 in \cite{HKR} for the general case.
\end{proof}

We include a here more careful study of $N_{G}(G_\R)$.
\begin{lm}\label{lm KA_uK}
Let $(G_{\R},H,\theta,B)<(G,U,\tau,B)$ be a real form of a complex reductive Lie group. Let $A_u=e^{i\la}\subset U$.
Then, we have a short exact sequence
$$
 1\to H^0A_uH^0\to U\to \pi_0(U)\to 1
$$
where $H^0$ denotes the connected component of $H$ and $\pi_0(U)$ is the group of connected components of $U$.
\end{lm}
\begin{proof} To see that $H^0A_uH^0=U^0$, we note that both subgroups have the same Lie algebra, as
$\lu=\h\oplus  i\m$ and $\m=\cup_{h\in H^0} \Ad(h)\la$ by Proposition 7.29 in \cite{K}, and clearly $H^0A_uH^0\subseteq U^0$.
\end{proof}
 

\section{The geometry of regular centralisers}\label{sec:reg_central}
{In this section we explain the main features about regular centraliser schemes for real groups. This yields to a natural description of cameral covers in terms of anisotropic Borel subgroups (cf. Proposition \ref{prop:alternative_cam_cover}).}

 Let  $\g_\R$ be a real reductive Lie subalgebra with complexification $\g$. Let $\theta$ denote the Cartan subalgebra of $\g_\R$ or its extension by complex linearization.  Let $\lt=\ld\oplus\la\subset \g$ be a $\theta$-invariant Cartan subalgebra of $\g$, with  $\la=\lt\cap\m$ maximal and $\ld=\lt\cap\h$. We let $S\subset \Delta=\Delta(\g,\lt)$ be the associated sets of (simple) roots, $W$ the Weyl group. Similarly, we can define $\Sigma(\la)$ the set of restricted roots associated to $\la$ (which is the image of the restriction map $res|_{\la}:\Delta\to\la^*$). This is a root system (possibly non reduced). Reflection with respect to simple such roots generates a group $W(\la)$ called the restricted Weyl group. This group is also isomorphic to
 $N_H(\la)/C_H(\la)$ and $N_{H_\R}(\la_{\R})/C_{H_\R}(\la_{\R})$, where $\la_\R=\la\cap\g_\R$. Moreover, any root $\alpha$ can be expressed as $\alpha=\lambda+i\beta$ with $\lambda\in\la_\R^*, \beta\in \ld_\R^*$.  
 
 Let $a=\dim\la$. Denote by $Ab^a(\m)$ the closed subvariety of $Gr(a,\m)$ whose points are abelian subalgebras
of $\m$. Define the incidence variety
\begin{equation}\label{eq mureg}
 \mu_{reg}=\{(x,\lc)\in \mr\times Ab^a(\m)\ :\ x\in\lc\}.
\end{equation}
We have the following.
\begin{prop}\label{basis}
The map 
\[
\psi: \mr\to Ab^a(\m)\ x\mapsto\lie{z}_{\m}(x)
\]
is smooth with smooth image and its graph is $\mu_{reg}$.
\end{prop}
\begin{proof}
First of all, note that the map is well defined: indeed, it is clear for regular and semsimimple elements in
$\m$. By Theorem 20 and Lemma 21 of \cite{KR71}, it extends to the whole of $\mr$.
As for smoothness, the proof of \cite[Proposition 1.3]{DG} adapts as follows.
\\We check that $\psi$ is well defined and has graph $\mu_{reg}$ by proving that $\mu_{reg}\to\mr$ is an embedding (hence, by properness  and surjectivity, an isomorphism).
To see this, as $\m_{reg}$ is reduced and irreducible (being a dense open set of a vector space), if the fibres are reduced points we will be done.
We have that
\[
T_{(x,\lie{b}_{ \m})}(\mu_{reg}\cap\{x\}\times Ab^a(\m))\cong\]
\[
\left\{T:\lie{b}\to\m/{\lie{b}}\left|\begin{array}{l} [T(y),x]=0\mbox{ for any }y\in\lie{b}\\
                                              T[y,z]=[Ty,z]+[y,Tz]\mbox{ for any }y,z\in\lie{b}
                                                \end{array}\right.
\right\}.
\]
By definition, the only $T$ satisfying those conditions is $T\equiv 0$, 
so the map is well defined. 
\\\\For smoothness, given a closed point $x\in\mr$, as $\mr\subset\m$ is open and dense, it follows that $T_x\mr\cong\m$. Consider
\[
\xymatrix{
T_x\mr\cong\m\ar[r]^{d_x\psi}&T_{\lie{z}_{\m}}Ab^a(\m)\ar[r]^{ev_x}&\m/\lie{z}_{\m}(x)\\
y\ar[r]&\{T:[T(z),x]=[-z,y]\}\ar[r]&T(x)=[y].
}
\]
Namely, $d_x\psi$ sends $y$ to the only map satisfying $[T(z),x]=[-z,y]$. Now, clearly $ev_x\circ d_x\psi$ is the projection map 
$\m\to\m/\lie{z}_{\m}(x)$. Also, $ev_x$ is surjective. We will prove it is injective, so it will follow that $Im(\psi)$ is contained in
 the smooth locus of $Ab^a(\m)$. The same fact proves that $d_x\psi$ must be surjective, and so we will be done.
\\Suppose $T(x)=T'(x)$ for some $T,\ T'\in T_{\lie{z}_{\m}(x)}Ab^a(\m)$. Then:
\[
0=[T(x)-T'(x),y]=[-x,T(y)]-[-x,T'(y)]=[-x,T-T'(y)] 
\]
for all $y\in\z_\m(x)$,  and hence $ev$ is injective.
\end{proof}
\begin{defi}\label{def:HN}
 We will call the image of $\psi$ the variety of regular centralisers, and denote it by $\ol{H/\NH}$.
\end{defi}
\begin{rk}\label{rk:open_HN}
$H/\NH\subset \ol{H/\NH}$ is an open subvariety consisting of the image of $\mrs$. This coincides with $\Gtheta/\Ntheta$, where $\Ntheta$ denotes the normaliser of $\laC$ in $\Gtheta$, since the latter group normalizes $\mr$ and semisimple orbits are the same for $H$ and $\Gtheta$.
\end{rk}
\begin{rk}
Note that $\psi$ is $\Gtheta$-equivariant for the isotropy representation on $\mr$ and conjugation on $Ab^a(\mr)$. 
\end{rk}
In what follows we present an alternative approach to cameral covers, following \cite{DG}.
\begin{defi}\label{def:min_aniso}
A parabolic subalgebra $\lp\subset \g$ is called minimal $\theta$-anisotropic if it is opposed to $\theta(\lp)$ and contains a $\theta$-invariant Cartan subalgebra with maximal intersection with $\m$. Namely, such objects are $\Gtheta$ or $H$ conjugate to parabolic subalgebras of the form
$$
\p= \laC\oplus\lc_{\h}(\la)\oplus\bigoplus_{\lambda\in \Lambda(\la)^+}\g_\lambda,
$$
where $\Lambda(\la)^+$ denotes a set of positive restricted roots. When $\g_{\R}\subset\g$ is quasi-split, since minimal $\theta$-anisotropic parabolic subalgebras are Borel subalgebras.
\end{defi}
\begin{prop}\label{prop:var_min_par}
The variety $H/\CH\cong\Gtheta/\CH_\theta$ parametrises minimal $\theta$-anisotropic parabolic subalgebras of $\g$.
\end{prop}
\begin{proof}
	The fact that $H/\CH$ parametrises minimal $\theta$-anisotropic parabolic subalgebras of $\g$ follows from \cite[Proposition 5]{Vust}. The isomorphism is a consequence of the morphism  exact sequences defined by the following commutative diagram:
	$$
	\xymatrix{
		\CH\ar[r]\ar[d]&\CH_\theta\ar[r]\ar[d]&F^2\ar@{=}[d]\\
		H\ar[r]&\Gtheta\ar[r]&F^2}
	$$
\end{proof}
\begin{prop}\label{prop:ref_cent_in_aniso}
	Any element $x\in\mr$ satisfies that $\lc_\m(x)\subset\p$ for some $\p\in H/\CH$.
\end{prop}
\begin{proof}
	By \cite{KR71}, any such element belongs to $\tilde{\g}$, where $\tilde{\g}\subset\g$ is a maximal split subalgebra, inside of which its centraliser remains the same. Thus, for some  Borel subalgebra $\lb\subset\tilde{\g}$, which may be easily chosen to be $\theta$-invariant, $\lc_\m(x)\subset\lb$. Since any such can be promoted to a minimal $\theta$-anisotropic parabolic subalgebra, we are done.
\end{proof}
\begin{defi}\label{def:HC}
	Define the group scheme $\HC$ to be the incidence variety inside $\HN\times H/\CH$. 
\end{defi}
\begin{prop}\label{prop:alternative_cam_cover}
	Let $G_{\R}<G$ be a quasi-split real form. Then, the choice of a $\theta$-anisotropic Borel subalgebra $\laC\subset\lb_0$ determines an isomorphism
	$$
\mr\times_{\HN}\HC\cong\mr\times_{\la/W(\la)}\laC.
	$$
\end{prop}
\begin{proof}
	By regularity of $\mr$ it is enough to prove that both schemes admit a morphism which makes them isomorphic over a dense open set. Define the morphism
	$$
	\mr\times_{\HN}\HC\ni(x,\lb)\mapsto (x,\pi_{\lb}(x))\in \mr\times_{\la/W(\la)}\laC
	$$
	where $\pi_{\lb}(x)\in\laC$ is the image of the class of $x$ in $\lb/[\lb,\lb]$ under the canonical isomorphism $\lb/[\lb,\lb]\cong\laC\oplus\lc_{\h}(\la)$, which is $\theta$ equivariant and so well defined.

	By Remark \ref{rk:open_HN}, over $H/\NH$, the above morphism is an isomorphism, as both schemes are $W(\la)$-principal bundles over $\mrs$.
	
\end{proof}

\end{document}